\documentclass[final,3p,times]{elsarticle}

\usepackage{amssymb}
\usepackage [latin1]{inputenc}
\usepackage{algorithm}
\usepackage{algorithmic}
\usepackage{multirow}
\usepackage{xcolor}
\usepackage{hyperref}
\usepackage{breakurl}
\usepackage{booktabs,threeparttable,caption2}
\usepackage{amsthm,amsmath}
\usepackage{float}
\usepackage{mathrsfs}
\usepackage{latexsym}
\usepackage{tabularx}
 \usepackage{diagbox}
 \usepackage{rotating}
\usepackage{cleveref}
 \biboptions{numbers,sort&compress}
\hypersetup{
    colorlinks=true,
    linkcolor=blue,
    filecolor=blue,
    urlcolor=blue,
    citecolor=cyan}
\newtheorem{thm}{Theorem}

\newtheorem{lem}{Lemma}

\newdefinition{example}{Example}
\newdefinition{rmk}{Remark}

\newcaptionstyle{left}{
\usecaptionmargin\captionfont
{\flushleft\bfseries\captionlabelfont\captionlabel\par}
\mbox{\onelinecaption{\captiontext}{\captiontext}}
}

\begin{document}

\begin{frontmatter}


\title{A low-rank algorithm for evaluating Lyapunov operator $\varphi$-functions within matrix-valued exponential integrators}
\author[author1]{Dongping Li\corref{cor1}}
\ead{lidp@ccsfu.edu.cn}
\author[author1]{Xiuying Zhang}
\ead{xiuyingzhang@ccsfu.edu.cn}
\author[author2]{Hongjiong Tian}
\ead{hjtian@shnu.edu.cn}
\cortext[cor1]{Corresponding author.}
\address[author1]{Department of Mathematics, Changchun Normal University, Changchun 130032, PR China}
\address[author2]{Department of Mathematics, Shanghai Normal University, Shanghai 200234, PR China}

\begin{abstract}
In this work we present a low-rank algorithm for computing low-rank approximations of 
large-scale Lyapunov operator $\varphi$-functions. These computations play a crucial role in implementing 
of matrix-valued exponential integrators for large-scale stiff matrix differential equations, where the (approximate) solutions are of low rank.
The proposed method employs a scaling and recursive procedure, complemented by a quasi-backward error analysis to determine the optimal parameters. The computational cost is primarily determined by the multiplication of sparse matrices with block vectors.
Numerical experiments validate the effectiveness of the proposed method as a foundational tool for matrix-valued exponential integrators 
in solving differential Lyapunov equations and Riccati equations.
\end{abstract}


\end{frontmatter}

\section{Introduction}\label{sec:1}

In this paper we consider the efficient implementation of matrix-valued exponential integrators for large-scale matrix differential equations (MDEs)
of the form
\begin{equation}\label{1.1}
\left\{
\begin{array}{l}
X'(t)=AX(t)+X(t)A^T+N(t,X(t)),~~\\
X(t_0)=X_0,
\end{array}
\right.
\end{equation}
where $A\in \mathbb{R}^{N\times N}$, $N:\mathbb{R}\times \mathbb{R}^{N\times N}\rightarrow \mathbb{R}^{N\times N}$ is the nonlinear part,
and $X(t)\in \mathbb{R}^{N\times N}$.
MDEs~(\ref{1.1}) play significant roles in various fields such as optimal control, model reduction, and
semi-discretization of two-dimensional parabolic partial differential equation (see e.g., \cite{Abou,Antoulas,Ascher}).
The well-known differential Lyapunov equations (DLEs) and differential Riccati equations (DREs) fall under this category.
In many applications, MDEs~(\ref{1.1}) exhibit low-rank inhomogeneous term and initial value, which are critical
features often efficiently utilized in large-scale computations.
Over the last few years, numerous time integration methods capable of exploiting the low-rank structure
have been proposed for solving large-scale DLEs, DREs and related problems,
including BDF, Rosenbrock methods, splitting methods and Krylov-based projection methods,
see, e.g., \cite{Behr,Jbilou18,Mena,Simoncini20,Koskela,Lang15,Stillfjord1,Stillfjord2,Stillfjord3}.

Exponential integrators have a rich history and are highly competitive with implicit methods for integrating stiff problems.
A primary feature of this class of integrators is that they treat the linear part exactly and the nonlinear part approximately.
This gives the integrators good stability properties and enables them to integrate stiff problems explicitly.
To date, numerous vector-valued exponential integration schemes have been developed (see, e.g., \cite{MH2006,MH2011,MH2009,Luan2013,Tokman2006}).
For a comprehensive overview of recent developments in exponential integrators, we refer the reader to  \cite{Hochbruck2010} and the references therein. Recently, low-rank matrix-valued exponential Rosenbrock-type integrators have been introduced as a competitive alternative for solving DREs \cite{Li2021}.  
For vector-valued exponential integrators, a critical component is the computation of 
matrix $\varphi$-functions during the implementation process. In contrast,  matrix-valued exponential integrators necessitate the evaluation of operator $\varphi$-functions at each time step.

For MDEs~(\ref{1.1}), matrix-valued exponential integrators can be constructed through the following reformulation of the problem as an integral equation.
Let $\mathcal{L}_A:\mathbb{R}^{N\times N}\rightarrow \mathbb{R}^{N\times N}$ denote the Lyapunov operator by
\begin{equation*}\label{1.1b}
\begin{aligned}
\mathcal{L}_A[X] = AX+XA^T, ~~ A\in \mathbb{R}^{N\times N}.
\end{aligned}
\end{equation*}
Using the variation-of-constants formula, the exact solution
of MDEs~(\ref{1.1}) can be represented as (see \cite{Behr})
\begin{equation}\label{1.1a}
\begin{aligned}
X(t_n+h)=e^{h\mathcal{L}_A}[X({t_n)}]+h\int_{0}^{1}e^{(1-\tau)h\mathcal{L}_A}[N(t_n+\tau h,X(t_n+\tau h))]\text{d}\tau.
\end{aligned}
\end{equation}
One then can derive the matrix-valued exponential integrators by interpolating the nonlinear term $N(t_n+\tau h,X(t_n+\tau h))$ in (\ref{1.1a}).
In the simplest case, we approximate the nonlinearity $N(t_n+\tau h,X(t_n+\tau h))$ by $N_n:=N(t_n,X(t_n))$
and thus obtain the matrix-valued exponential time integration, known as exponential Euler scheme (denoted by \texttt{mExpeul})
\begin{eqnarray}\label{1.1c}
X_{n+1}=e^{h\mathcal{L}_A}[X_n]+h\varphi_{1}(h\mathcal{L}_A)[N_n], 
\end{eqnarray}
where
\begin{eqnarray*}\label{1.1d}
\varphi_1(z)=\int_{0}^{1}e^{(1-\theta)z}d\theta=1+\frac{z}{2!}+\frac{z}{3!}+\cdots.
\end{eqnarray*}
If $N(t,X(t))$ is a constant matrix, scheme (\ref{1.1c}) provides the exact solution of Eq. (\ref{1.1}). 

Scheme (\ref{1.1c}) requires the computation of the actions of the operator exponential and $\varphi_{1}$ at each time step of the integration. 
For more general matrix-valued exponential integrators,  the central challenge is to evaluate the actions of general $\varphi$-functions of operators on specific matrices:
\begin{equation}\label{1.2}
\varphi_l(\mathcal{L}_A)[Q],~~l\in \mathbb{N},
\end{equation}
where $A\in \mathbb{R}^{N\times N}$ is large and sparse,  $Q\in \mathbb{R}^{N\times N}$ is symmetric and of low rank, and functions $\varphi_{l}(z)$ are defined as
\begin{equation*}\label{1.3}
\varphi_{l}(z)=
\begin{cases}
e^{z}, & l=0,\\
\frac{1}{(l-1)!}\int_{0}^{1}e^{(1-\theta)z}\theta^{l-1}d\theta, & l\ge 1.\\
\end{cases}
\end{equation*}
The operator $\varphi$-functions can be rewritten as Taylor series expansion
\begin{equation*}\label{1.4}
\varphi_l(\mathcal{L}_A)=\sum\limits^{\infty}_{k=0}\frac{1}{(k+l)!}{\mathcal{L}_A}^k, ~~l\in \mathbb{N},
\end{equation*}
where ${\mathcal{L}_A}^k$ denotes the $k$-fold composition of the Lyapunov operator $\mathcal{L}_A$.
In particular, ${\mathcal{L}_A}^0=I_N$, here $I_N$ is the $N\times N$ identity matrix. These operator functions satisfy the recursive relation
\begin{equation}\label{1.5}
\varphi_k(\mathcal{L}_A)=\mathcal{L}_A\varphi_{k+1}(\mathcal{L}_A)+\frac{1}{k!}I_N,~~k\geq 0.
\end{equation}
Furthermore, we have
\begin{equation}\label{1.6}
\varphi_k(\mathcal{L}_A)=\mathcal{L}_A^{-k} \left(e^{\mathcal{L}_A}-\sum\limits^{k-1}_{j=0}\mathcal{L}_A^{j}/j!\right),~~k\geq 0.
\end{equation}
For a singular operator $\mathcal{L}_A$, this formula in (\ref{1.6}) is interpreted by expanding its right-hand side of as a power series in $\mathcal{L}_A$.

Besides the matrix-valued exponential integrators, operator functions like (\ref{1.2}) also appear in finite-time controllability and observability Grammians of linear control
systems, stochastic differential equations and filtering theory, see \cite{Arnold74,Jazwinski}.
Evaluating these operator functions numerically is not a trivial task. Van loan \cite{Loan} proposed to approximate (\ref{1.2})
by computing a single exponential of a larger block upper triangular matrix, a further generalization in this direction can be found in \cite{Carbonell08}.
However, this approach may suffer from
an overflow error from a computational viewpoint and is difficult to exploit the low-rank structure of $Q$.
In \cite{Li2021}, the standard numerical quadrature formula is utilized to solve (\ref{1.2}); however, the approach may be time-consuming to obtain highly accurate results while less accurate approximation may lead to numerical instability during
the implementation of exponential integrators. Recently, a modified scaling and squaring method has been presented
for dense and moderate-sized problems in \cite{Li2023}, but its application in large-scale problems is limited.
The aim of the present paper is to propose a low-rank method to compute these operator functions.
By exploiting the sparsity of $A$ and the low-rank representation of $Q$, we present low-rank
approximation based on an $LDL^T$-type decomposition. The method uses a scaling and
recursive procedure and can be implemented by matrix-vectors products to reduce the storage and computational complexities.

The paper is organized as follows. In Section~\ref{sec:2}, we briefly introduce the scaling and
recursive procedure for (\ref{1.2}) and show how to choose the optimal parameters to achieve the required accuracy.
In  Section~\ref{sec:3}, we exploit the $LDL^T$-based algorithm, which constitutes the primary focus of this paper.
Numerical experiments in Section~\ref{sec:4} demonstrate the performance of the proposed method and illustrate that it
can be used as the basis of matrix-valued exponential integrators for solving
large-scale DLEs and DREs. Finally, we draw some conclusions in  Section~\ref{sec:5}.

Throughout the paper $\|\cdot\|$ refers to any consistent matrix or operator norm, in particular $\|\cdot\|_1$ and $\|\cdot\|_F$ denote
the 1-norm and the Frobenius norm, respectively. We denote by $\text{Lyap}(N)$ the set of Lyapunov operators $\mathcal{L}_A$
for any $A\in \mathbb{R}^{N\times N}.$ The notation $\rho(\cdot)$ denotes the spectral radius of a matrix or operator, and
$\otimes$ represents the Kronecker product of matrices.
$\lceil x\rceil$ denotes the smallest integer not less than $x$. Matlab-like notations are used whenever necessary.

\section{Full-rank computing}
\label{sec:2}
\subsection{The scaling and recursive procedure}
\label{subsec:a}
If $\|\mathcal{L}_A\|$ is sufficiently small, $\varphi_l\left(\mathcal{L}_A\right)[Q]$ can be directly evaluated using either a polynomial or rational approximation. However, for large $\|\mathcal{L}_A\|$, this approach becomes impractical. In this section we introduce a scaling and recursive procedure for the computation of operator $\varphi$-functions. 
We begin by recalling a general formula for $\varphi$-functions \cite{Skaflestad}, with its proof to Lyapunov operators provided in \cite{Li2023}.

For a fixed $l\in \mathbb{N}$ , the result states that
\begin{equation*}\label{2.1}
(a+b)^l\varphi_l\left((a+b)\mathcal{L}_A\right)=a^l\varphi_0(b\mathcal{L}_A)\varphi_l(a\mathcal{L}_A)
+\sum\limits^{l}_{j=1}\frac{a^{l-j}b^{j}}{(l-j)!}\varphi_{j}(b\mathcal{L}_A),~~~a, b \in \mathbb{R}.
\end{equation*}
In particular, setting $a=k-1$ and $b=1$ for an integer $k\geq 2 $, we have
 \begin{eqnarray}\label{2.2}
\varphi_l\left(k\mathcal{L}_A\right)=(1-\frac{1}{k})^l\varphi_0(\mathcal{L}_A)\varphi_l\left((k-1)\mathcal{L}_A\right)+\sum\limits^{l}_{j=1}\mu_{k,j}\varphi_j(\mathcal{L}_A),
\end{eqnarray}
where
\begin{eqnarray*}\label{2.3}
\mu_{k,j}:=(1-\frac{1}{k})^l(\frac{1}{k-1})^j\frac{1}{(l-j)!}.
\end{eqnarray*}
The identity (\ref{2.2}) can be used as a starting in derivation of the scaling and recursive procedure for solving (\ref{1.2}).

Let $s$ be a non-negative integer; define  $X:=A/s$, and let $\mathcal{L}_{X}$ be the Lyapunov operator generated by matrix $X$. It follows that
$\mathcal{L}_{X}=\mathcal{L}_A/s$. Furthermore, we define
\begin{eqnarray}\label{2.3b}
&&C_k:=\sum\limits^{l}_{j=1}\mu_{k,j}\varphi_j(\mathcal{L}_{X})[Q],\\
&&\Phi_k:=\varphi_l(k\mathcal{L}_{X})[Q].
\end{eqnarray}
Then, starting from $\Phi_1=\varphi_l(\mathcal{L}_{X})[Q]$, $\varphi_l(\mathcal{L}_A)[Q]$ can be computed iteratively using the recurrence relation:
\begin{eqnarray}\label{2.4}
\Phi_k=(1-\frac{1}{k})^l\varphi_0(\mathcal{L}_{X})[\Phi_{k-1}]+C_k,~~k=2,3,\cdots,s.
\end{eqnarray}
To numerically implement the above recursion, the following steps are required: 

(i) Precompute $\varphi_j(\mathcal{L}_{X})[Q]$ for $j=1,2,\cdots,l$, and subsequently determine $\Phi_1$ and $C_k$ for $k=2,3,\cdots,s$; 

(ii) Develop a method to implement $\varphi_0(\mathcal{L}_{X})[\cdot]$ involved in (\ref{2.4}).

For the scaled operator $\mathcal{L}_{X}$, task (i) can be accomplished using polynomial approximations (e.g., Taylor, Chebyshev or Hermite approximations) or rational approximations. Methods based on rational approximations require solving algebraic Lyapunov matrix equations, which are generally computationally expensive for large-scale problems. For simplicity, we focus on the Taylor approximation in this work. 

First, we approximate $\varphi_l(\mathcal{L}_{X})[Q]$ using its truncated Taylor series of degree $m$, given by
\begin{eqnarray}\label{2.5}
\varphi_l(\mathcal{L}_{X})[Q]\approx\sum\limits^{m}_{k=0}\frac{1}{(k+l)!}{\mathcal{L}_{X}}^k[Q]\equiv T_{l,m}(\mathcal{L}_{X})[Q].
\end{eqnarray}
The operator polynomial $ T_{l,m}(\mathcal{L}_{X})[Q]$ can be computed through a series of Lyapunov operators, as described in Algorithm ~\ref{alg2.0}.
In practice, the scaling parameter $s$ should be chosen such that the norm of $\mathcal{L}_{X}=\mathcal{L}_{A}/s$ is sufficiently small. This ensures that  $\varphi_l(\mathcal{L}_{X})$ can be accurately approximated using its truncated Taylor series with an appropriatly chosen degree $m$. The detailed criteria for  selecting these two parameters will be discussed in the next subsection. Next, using the recursive relation (\ref{1.5}), we compute $\varphi_j(\mathcal{L}_{X})[Q]$ for $j=l-1,l-2,\ldots,1$ as follows:
\begin{equation}\label{2.6}
T_{j,m}(\mathcal{L}_{X})[Q]:=\mathcal{L}_{X}\left [T_{j+1,m}(\mathcal{L}_{X})[Q]\right]+\frac{1}{j!}Q,~ j=l-1,l-2,\cdots,1.
\end{equation}
Clearly, $T_{j,m}(\mathcal{L}_{X})$ represents the degree $m+l-j$ truncated Taylor series of $\varphi_j(\mathcal{L}_{X})$.
Once $T_{j,m}(\mathcal{L}_{X})[Q]$ is computed, the matrices $C_k$ in (\ref{2.3b}) can be obtained by
substituting $\varphi_j(\mathcal{L}_{X})[Q]$ with $T_{j,m}(\mathcal{L}_{X})[Q]$.
\begin{algorithm}
\caption{The recursive procedure for computing operator polynomial (\ref{2.5}). }
\label{alg2.0}
\begin{algorithmic}[1]
\REQUIRE{$\mathcal{L}_{X}$, $Q$, $m$ and $l$}
\STATE{Define $P:=\frac{1}{l!}Q$}
\STATE {Define $S:=P$}
\FOR{$k=1:m$}
\STATE  {Update $P:=\frac{1}{(l+k)}\mathcal{L}_{X}[P]$}
\STATE   {Update $S:=S+P$}
\ENDFOR
\ENSURE~$S$
\end{algorithmic}
\end{algorithm}

Task (ii) can be  accomplished by approximating $\varphi_0(\mathcal{L}_{X})$ using truncated Taylor series of degree $m+l$, given by
\begin{equation}\label{2.7}
 \varphi_0(\mathcal{L}_{X})\approx \sum\limits^{m+l}_{k=0}\frac{{\mathcal{L}_{X}}^k}{k!}\equiv T_{0,m}(\mathcal{L}_{X}).
\end{equation}

In the following, we use the notations $\widehat{\Phi}_k$ and $\widehat{C}_k$ to represent the 
approximations of $\Phi_k$ and $C_k$, respectively. Algorithm~\ref{alg2.1} outlines a general procedure for computing $\varphi_l(\mathcal{L}_{X})[Q]$.
The computational cost of the algorithm is primarily determined by the matrix-matrix multiplications. Specifically, the total number of such multiplications is $m$ when $s=1$, and $s(m+l)$ when $s>1.$

\begin{algorithm}
\caption{The scaling and recursive procedure for computing $\varphi_l(\mathcal{L}_A)[Q]$.}
\label{alg2.1}
\begin{algorithmic}[1]
\REQUIRE~{$A, Q \in \mathbb{R}^{N\times N},$ $l$}
\STATE {Select optimal values of $m$ and $s$}
\STATE  {Compute $X=\frac{1}{s}A$}
\STATE  {Compute $B_{l}=T_{l,m}(\mathcal{L}_{X})[Q]$ using Algorithm~\ref{alg2.0}}
\IF {$s=1$} \RETURN $\widehat{\Phi}_s=T_{l}$ \ENDIF 
\FOR{$k=l-1:-1:1$}
\STATE  {Compute $T_{k}=\mathcal{L}_{X}[B_{k+1}]+\frac{1}{k!}Q$}
\ENDFOR
\STATE  {Set $\widehat{\Phi}_1=B_{l}$}
\FOR{$k=2:s$}
\STATE  {Compute $\widehat{C}_k=\sum\limits^{l}_{j=1}\mu_{k,j}B_{j}$ with $\mu_{k,j}=(1-\frac{1}{k})^{l-j}(\frac{1}{k})^j\frac{1}{(l-j)!}$}
\STATE  {Compute $\widehat{\Phi}_k$ by the recurrence $\widehat{\Phi}_k=(1-\frac{1}{k})^lT_{0,m}(\mathcal{L}_{X})[\widehat{\Phi}_{k-1}]+\widehat{C}_k$}
\ENDFOR
\ENSURE{$\widehat{\Phi}_s$}
\end{algorithmic}
\end{algorithm}

\subsection{Choice of the parameters $m$ and $s$}
\label{sec:2.2}
Algorithm \ref{alg2.1} involves two key parameters: the polynomial degree $m$ of the polynomial $T_{l,m}(z)$
and the scaling parameter $s$, both of which must be chosen appropriately.
Following the backward error analysis developed in \cite{AlMohy2009,AlMohy2011,Higham2005},
we present a quasi-backward error analysis for Algorithm \ref{alg2.1}. This analysis  serves as the theoretical basis for determining parameters $m$ and $s.$
 
 Define the Lyapunov operator set
\begin{equation*}\label{2.8}
\Omega_m:=\{ \mathcal{L}:~~\|e^{-\mathcal{L}}T_{0,m}(\mathcal{L})-I_N \|<1,~\mathcal{L}\in \text{Lyap}(N)\}.
\end{equation*}
Note that
\begin{equation*}\label{2.8a}
e^{-\mathcal{L}}T_{0,m}(\mathcal{L})-I_N=-e^{-\mathcal{L}}\sum\limits^\infty_{k=m+l+1}\frac{1}{k!}{\mathcal{L}}^{k},
\end{equation*}
where $T_{0,m}(\cdot)$ is defined as in (\ref{2.7}).
Then, the operator function
\begin{equation}\label{2.9}
h_{m+l}(\mathcal{L}):=\log\left(e^{-\mathcal{L}}T_{0,m}(\mathcal{L})\right)
\end{equation}
exists over $\Omega_m$ and it has the Taylor series
\begin{equation*}\label{2.10}
h_{m+l}(\mathcal{L})=\sum\limits^\infty_{k=m+l+1}c_k{\mathcal{L}}^{k}.
\end{equation*}
Here, $\log$ denotes the principal logarithm  function. For a given fixed value of $m+l$, the polynomial $h_{m+l}(z)$ can be determined directly 
by expanding the logarithm function (\ref{2.9}) into a power series utilizing MATLAB's Symbolic Math Toolbox.
We now present a quasi-backward error analysis for Algorithm \ref{alg2.1}.

\begin{thm}\label{th1}Assume that $s^{-1}\mathcal{L}_A\in \Omega_m$ and the inverse of $\mathcal{L}_A$ exists. Then the approximation $\widehat{\Phi}_s$ generated by Algorithm \ref{alg2.1} satisfies
\begin{eqnarray*}\label{2.11}
\widehat{\Phi}_s={\mathcal{L}_A}^{-l}\left(e^{\mathcal{L}_A+\Delta \mathcal{L}_A}-\sum\limits^{l-1}_{j=0}{\mathcal{L}_A}^j/j!\right)[Q],
\end{eqnarray*}
where
\begin{equation*}
\Delta \mathcal{L}_A:=sh_{m+l}(s^{-1}\mathcal{L}_A).
\end{equation*}
\end{thm}

\begin{proof} We prove the claim for $\widehat{\Phi}_k$ by finite induction over $k\in \{1,2,\ldots,s\}.$
From (\ref{2.9}) it follows that
\begin{equation*}\label{2.12}
T_{0,m}(\mathcal{L}_{X})=e^{\mathcal{L}_{X}+h_{m+l}(\mathcal{L}_{X})},
\end{equation*}
where $\mathcal{L}_{X}=s^{-1}\mathcal{L}_A.$ Furthermore, using (\ref{2.6}) we obtain
\begin{equation*}\label{2.13}
T_{i,m}(\mathcal{L}_{X})=\mathcal{L}_{X}^{-i}\left(e^{\mathcal{L}_{X}+h_{m+l}(\mathcal{L}_{X})}-\sum\limits^{i-1}_{j=0}{\mathcal{L}_{X}}^j/j!\right),~i=1,2,\cdots,l.
\end{equation*}
Notice that $\widehat{\Phi}_{1}=T_{l,m}(\mathcal{L}_{X})[Q].$ This gives the base case.

Now we assume that
\begin{equation*}\label{2.14}
\widehat{\Phi}_{k-1}=\left((k-1)\mathcal{L}_{X}\right)^{-l}\left(e^{(k-1)(\mathcal{L}_{X}+h_{m+l}(\mathcal{L}_{X}))}-\sum\limits^{l-1}_{j=0}(k-1)^j{\mathcal{L}_{X}}^j/j!\right)[Q],~~k\geq2.
\end{equation*}
The inductive step follows from

\begin{equation}\label{2.15}
\begin{aligned}
\widehat{\Phi}_k&=(1-\frac{1}{k})^lT_{0,m}(\mathcal{L}_{X})[\widehat{\Phi}_{k-1}]+\widehat{C}_k\\
&=\left(k\mathcal{L}_{X}\right)^{-l}\left(e^{k(\mathcal{L}_{X}+h_{m+l}(\mathcal{L}_{X}))}-e^{\mathcal{L}_{X}+h_{m+l}(\mathcal{L}_{X})}\sum\limits^{l-1}_{j=0}(k-1)^j{\mathcal{L}_{X}}^j/j!\right)[Q]\\
&+\sum\limits^{l}_{j=1}\mu_{k,j}{\mathcal{L}_{X}}^{-j}\left(e^{\mathcal{L}_{X}+h_{m+l}(\mathcal{L}_{X})}-\sum\limits^{j-1}_{i=0}{\mathcal{L}_{X}}^i/i!\right)[Q]\nonumber\\
&=\left(k\mathcal{L}_{X}\right)^{-l}\left(e^{k(\mathcal{L}_{X}+h_{m+l}(\mathcal{L}_{X}))}-e^{\mathcal{L}_{X}+h_{m+l}(\mathcal{L}_{X})}\sum\limits^{l-1}_{j=0}\frac{(k-1)^j}{j!}{\mathcal{L}_{X}}^j\right.\\
&+\left.\sum\limits^{l}_{j=1}\frac{(k-1)^{l-j}}{(l-j)!}{\mathcal{L}_{X}}^{l-j}\left(e^{\mathcal{L}_{X}+h_{m+l}(\mathcal{L}_{X})}-\sum\limits^{j-1}_{i=0}{\mathcal{L}_{X}}^i/i!\right)\right)[Q]\\
&=\left(k\mathcal{L}_{X}\right)^{-l}\left(e^{k(\mathcal{L}_{X}+h_{m+l}(\mathcal{L}_{X}))}
-\sum\limits^{l}_{j=1}\sum\limits^{j-1}_{i=0}\frac{(k-1)^{l-j}}{(l-j)!i!}{\mathcal{L}_{X}}^{l+i-j}\right)[Q]\\
&=(k\mathcal{L}_{X})^{-l}\left(e^{k\left(\mathcal{L}_{X}+h_{m+l}(\mathcal{L}_{X})\right)}
-\sum\limits^{l-1}_{j=0}\sum\limits^{j}_{i=0}\frac{(k-1)^{j-i}}{(j-i)!i!}{\mathcal{L}_{X}}^j\right)[Q]\\
&=(k\mathcal{L}_{X})^{-l}\left(e^{k\left(\mathcal{L}_{X}+h_{m+l}(\mathcal{L}_{X})\right)}-\sum\limits^{l-1}_{j=0}(k\mathcal{L}_{X})^j/j!\right)[Q].
\end{aligned}
\end{equation}
The result now follows by taking $k=s$ and replacing $\mathcal{L}_{X}$ by $\mathcal{L}_A/s$.
\end{proof}

Theorem~\ref{th1} shows that the approximation produced by Algorithm~\ref{alg2.1} can be regarded as a perturbation of $\mathcal{L}_A$ in the operator exponential given in (\ref{1.6}). In particular, when $l=0$, Theorem~\ref{th1} becomes a backward error analysis for computing $\text{exp}(\mathcal{L}_{X})[Q]$.

Given a tolerance \texttt{Tol}, we aim to select appropriate parameters $m$ and $s$ such that
\begin{equation}\label{2.16}
\frac{\|\Delta \mathcal{L}_A\|}{\|\mathcal{L}_A\|}=\frac{\|h_{m+l}(s^{-1}\mathcal{L}_A)\|}{\|s^{-1}\mathcal{L}_A\|}\leq\texttt{Tol}.
\end{equation}
Define the function
\begin{equation*}\label{2.17a}
\bar{h}_{m+l}(x):=\sum\limits^\infty_{k=m+l}|c_{k+1}|x^{k},
\end{equation*}
and let
\begin{equation*}\label{2.17}
\theta_{m+l}=\max{\{\theta:{\bar{h}_{m+l}(\theta)}\leq \texttt{Tol}\}}.
\end{equation*}
Then, once we choose $s$ such that
\begin{equation*}\label{2.18}
s^{-1}\|{\mathcal{L}_A}^{k}\|^{1/k} \leq \theta_{m+l}~~\text{for}~~k\geq m+l,
\end{equation*}
the relative quasi-backward error (\ref{2.16}) satisfies
\begin{equation*}\label{2.19}
\frac{\|\Delta \mathcal{L}_A\|}{\|\mathcal{L}_A\|}\leq \bar{h}_{m+l}(\theta_{m+l}) \leq \texttt{Tol}.
\end{equation*}
In practice, we compute $\theta_{m+l}$ by replacing $\bar{h}_{m+l}(x)$ with its truncated $\nu$-terms series and solving numerically the corresponding algebra equation
\begin{equation*}\label{2.17b}
\sum\limits^{\nu+m+l}_{k=m+l}|c_{k+1}|x^{k}=\texttt{Tol}.
\end{equation*}
Table~\ref{tab2.1} lists the evaluations of $\theta_{m+l}$ for some $m+l$ when $\nu=150$ and $\texttt{Tol}=2^{-53}$.
\begin{table}[h]
\caption{The values of $\theta_{m+l}$ for selected $m+l$ when $\texttt{Tol}=2^{-53}$.}\label{tab2.1}%
\begin{tabular}{@{}l|lllllllllll@{}}
\toprule
$m+l$ & $5$& $10$& $15$& $20$& $25$& $30$& $35$ & $40$ & $45$ & $50$ & $55$\\
\midrule
$\theta_{m+l}$ & $2.40\text{e-}3$ & $1.44\text{e-}1$ & $6.41\text{e-}1$ &$1.44\text{e}0$&$2.43\text{e}0$
  &$3.54\text{e}0$ &$4.73\text{e}0$&$5.97\text{e}0$ &$7.25\text{e}0$&$8.55\text{e}0$ &$9.87\text{e}0$ \\
\bottomrule
\end{tabular}
\end{table}

To determine the range of $s$,  it is necessary to establish an upper bound for $\|\mathcal{L}_A^{k}\|^{1/k}$. As shown in \cite{Li2023},  two approaches are available.
The first approach involves bounding the norm of $\mathcal{L}_A^{k}$ using the spectral radius of $\mathcal{L}_A$. For any $\epsilon>0$,
there exists a consistent norm $\|\cdot\|_\epsilon$ such that $\|\mathcal{L}_A\|_\epsilon \leq 2\rho (A)+\epsilon.$ This result implies
\begin{equation*}\label{2.20}
\|{\mathcal{L}_A}^k\|_\epsilon^{1/k} \leq 2\rho (A)+\epsilon.
\end{equation*}
We then set  $s=\max \{1, \lceil (2\rho (A)+\epsilon)/\theta_{m+l}~\rceil\}.$ 
The drawback of this approach is the challenge in representing such a norm.

The second approach is applicable to any consistent matrix norm. using the techniques outlined in the proof of Theorem 4.2 in \cite{AlMohy2009}, 
it can be readily verified that
\begin{equation*}\label{2.23}
\|{\mathcal{L}_A}^k\|^{1/k} \leq \alpha_p(\mathcal{L}_A): =\max\left(\|{\mathcal{L}_A}^{p}\|^{1/p},\|{\mathcal{L}_A}^{p+1}\|^{1/(p+1)}\right)~\text{for}~~p(p-1)\leq m+l\leq k.
\end{equation*}
The approach requires evaluating $\|{\mathcal{L}_A}^{p}\|^{1/p}$, $\|{\mathcal{L}_A}^{p+1}\|^{1/(p+1)}$ for $p(p-1)\leq m+l$.
Utilizing the binomial theorem, it can be shown that
\begin{equation}\label{2.28}
{\mathcal{L}_A}^p[X]=\sum\limits^{p}_{j=0}\binom{p}{j}A^jX(A^{p-j})^T,
\end{equation}
where $\binom{p}{j}=\frac{p!}{j!(p-j)!}$. Since $\sum\limits^{p}_{j=0} \binom{p}{j}=2$, we obtain
\begin{equation*}\label{2.29}
\|{\mathcal{L}_A}^p\|^{1/p}\leq d_p:=2\max\{\|A^j\|^{1/p}\cdot\|A^{p-j}\|^{1/p},~~ j=0,1,\ldots, p\}.
\end{equation*}
In this case, the scaling factor $s$ is naturally set to
\begin{equation*}\label{2.26}
s = \max\{1, \lceil \alpha_p(\mathcal{L}_A)/\theta_{m+l}\rceil\}.
\end{equation*}

In practice, specific values of $m$ and $p$ would be chosen to minimize the computational cost, satisfying
\begin{equation*}\label{2.27}
[m_*,p_*]=\arg \min\limits_{m,p}\{s(m+l):~~2\leq p\leq p_{\max}, ~~p(p-1)\leq m+l \leq m_{\max}\},
\end{equation*}
where $m_{\max}$ is the maximum allowable value of $m+l$, and $p_{\max}$ is the largest value of $p$ for which $p(p-1)\leq m_{\max}$.
For practical implementation, we set $m_{\max}=55$ and $p_{\max}=7.$
The procedure for selecting these values is outlined in Algorithm~\ref{alg2.2}. 
We use the 1-norm to evaluate the powers of $A$, which can be evaluated using the block 1-norm estimation algorithm \cite{Higham00}.

\begin{algorithm}[htb]
\caption{$\texttt{select\underline{~}m\underline{~}s}$: this algorithm computes the parameters $m$ and $s$.}\label{alg2.2}
\begin{algorithmic}[1]
\REQUIRE{$A\in \mathbb{R}^{N\times N},$ $l$, $m_{\max}$ and $p_{\max}$}
\FOR {$p = 2:p_{\max}+1$}
\STATE {Estimate $d_p =\max{\left\{\|A^k\|_1\cdot\|A^{p-k}\|_1,~k=0,1,\ldots,p\right\}}$}
\ENDFOR
\STATE {Compute $\alpha_p=2\max (d_p^{1/p}, ~d_{p+1}^{1/{(p+1)}}),$~$p=1,2,\cdots,p_{\max}$}\vspace{1ex}
\STATE {Compute  $[m_*,p_*]=\arg\min\limits_{m,p}\{(m+l)\lceil\alpha _p/\theta_{m+l}\rceil:~~ p(p-1)\leq m+l \leq m_{\max}\}$}\vspace{1ex}
\STATE {Set $m:=m_*$, $s:=\max(\lceil\alpha _{p_*}/\theta_{m_*+l}\rceil,1)$}
\ENSURE{$m,~s$}
\end{algorithmic}
\end{algorithm}

\section{low-rank implementation} \label{sec:3}
We now investigate the low-rank variant of Algorithm~\ref{alg2.1},
 motivated by  the frequently observed small (numerical) rank of $\varphi_l(\mathcal{L}_A)[Q]$,
given that $Q$ has a low-rank factorization of the form
 \begin{equation*}\label{3.0}
 Q=LDL^T, ~L\in \mathbb{R}^{N\times r}, D\in \mathbb{R}^{r\times r}, r\ll N.
 \end{equation*}
Utilizing the sparsity of $A$ and the low-rank structure of $Q$, we can efficiently implement
Algorithm \ref{alg2.1} in a low-rank approach, thereby reducing computational cost.
We will follow the notations introduced in Section~\ref{sec:2}.

To construct  a low-rank approximation, we need to provide low-rank representations of $B_k$, $\widehat{C}_k$ and $\widehat{\Phi}_k$
within Algorithm~\ref{alg2.1}, respectively. In considering how to compute the low-rank representation of $B_k$, we first state a result for the $LDL^T$ factorization of general Lyapunov operator polynomials.

\begin{lem}\label{lem1} Let $P_m(x)=\sum\limits^{m}_{k=0}a_kx^k,$ $\mathcal{L}_A\in \text{Lyap}(N)$~and~$Q=LDL^T\in \mathbb{R}^{N\times r}$
with $L\in \mathbb{R}^{N\times r}$ and~$D\in \mathbb{R}^{r\times r}.$  Then $P_m(\mathcal{L}_A)[Q]$ can be represented in the form of $LDL^T$-type as:
\begin{equation}\label{3.1}
P_m(\mathcal{L}_A)[Q]= \widetilde{L}(\Gamma\otimes D)\widetilde{L}^T,
\end{equation}
where $\widetilde{L}:= [L,AL,A^2L\ldots,A^mL]\in \mathbb{R}^{N\times r(m+1)}$, and $\Gamma \in \mathbb{R}^{(m+1)\times (m+1)}$ is a symmetric matrix, whose elements are zero below the first anti-diagonal and its non-zero entries are $\gamma_{i+1,j+1}=a_{i+j}\binom{i+j}{i}$ for $0\leq i+j \leq m, ~ i, j=0,1,\ldots,m$.
\end{lem}
\begin{proof}
Substituting the expression (\ref{2.28}) into the operator polynomial function $P_m(\mathcal{L}_A)[Q]$ and setting $a_k:=0$ for $k=m+1,m+2,\ldots,2m$, we have
\begin{eqnarray*}\label{3.2}
\begin{array}{ll}
P_m(\mathcal{L}_A)[Q]&=\sum\limits^{m}_{k=0}a_k{\mathcal{L}_A}^k[LDL^T]\vspace{1ex}\\
&=\sum\limits^{m}_{k=0}\sum\limits^{k}_{i=0}a_k\binom{k}{i}A^iLDL^T(A^T)^{k-i}\vspace{1ex}\\
&=\sum\limits^{m}_{i=0}\sum\limits^{m-i}_{j=0}a_{i+j}\binom{i+j}{i}A^iLD(A^jL)^T\vspace{1ex}.
\end{array}
\end{eqnarray*}
This directly yields (\ref{3.1}).
\end{proof}

Applying Lemma ~\ref{lem1} to $B_l=T_{l,m}(\mathcal{L}_{X})[LDL^T]$, the decomposition $\widetilde{L}_l\widetilde{D}_l{\widetilde{L}_l}^T$ is  given by the factors
\begin{equation*}\label{3.4}
\widetilde{L}_l= [L, XL, X^2L,\ldots, X^mL]~~ \text{and}~~ \widetilde{D}_l=(\Gamma\otimes D),
\end{equation*}
where the  entries of $\Gamma$ are:
\begin{equation}\label{3.5}
\gamma_{i+1,j+1} :=
\begin{cases}
\frac{(i+j)!}{i!j!(l+i+j)!}, &0\leq i+j \leq m,\\
0, & \text{otherwise}.
\end{cases}
\end{equation}
The column size of $\widetilde{L}_l$ is $r(m+1)$, which grows linearly with respect to the value
of $m$ and may far larger than the rank of $B_l$. This will produce additional costs in subsequent calculations.
Therefore, it is necessary to embed a column-compression strategy to eliminate the redundant information of $\widetilde{L}_l$ and $\widetilde{D}_l$,
as discussed in \cite[Section 3.4]{Lang15}.

Utilizing the recursion (\ref{2.6}) and starting with the low-rank factors $\widetilde{L}_l$ and $\widetilde{D}_l$ of $B_l$ ,
the low-rank factors $\widetilde{L}_k$, $\widetilde{D}_k$ of $B_k$ for $k=l-1,l-2,\ldots,1$, are given recursively by
\begin{equation*}\label{3.7}
\widetilde{L}_k= \begin{pmatrix}L,\widetilde{L}_{k+1},X\widetilde{L}_{k+1}\end{pmatrix}
\quad\text{and}\quad
\widetilde{D}_k=\begin{pmatrix}
\frac{1}{k!}D & 0& 0\\
0 & 0& \widetilde{D}_{k+1}\\
0 & \widetilde{D}_{k+1}& 0\\
\end{pmatrix}.
\end{equation*}
Collecting the above expressions, we obtain
\begin{eqnarray*}\label{3.8}
\widehat{C}_k=\sum\limits^{l}_{j=1}\mu_{k,j}B_j=G_kS_kG_k^T
\end{eqnarray*}
with
\begin{eqnarray*}\label{3.9}
G_k=[L_1, L_2, \cdots, L_l] ~~\text{and}~~S_k=\text{blkdiag}(\mu_{k,1}D_1, \mu_{k,2}D_2, \cdots, \mu_{k,l}D_l).
\end{eqnarray*}
In our computation, once the new splitting factors are formed, column-compression strategy will be performed to eliminate the redundant information.

The next step is to derive a low-rank approximation for $\widehat{\Phi}_k$.
From Line 13 of Algorithm~\ref{alg2.1}, we have observed that the key ingredient for the factorization
 is to form a low-rank approximation to $(1-\frac{1}{k})^lT_{0,m}(\mathcal{L}_{X})[\widehat{\Phi}_{k-1}].$
Assume the previous approximation $\widehat{\Phi}_{k-1}$
 admits a decomposition of the form
\begin{eqnarray*}\label{3.10}
\widehat{\Phi}_{k-1} = L_{k-1}D_{k-1}L_{k-1}^T
\end{eqnarray*}
with $L_{k-1}\in \mathbb{R}^{N\times r_{k-1}}, D_{k-1}\in \mathbb{R}^{r_{k-1}\times r_{k-1}}$.
Since $T_{0,m}(\mathcal{L}_{X})$ is the order $m+l$ truncated Taylor series of $e^{\mathcal{L}_{X}}$,
again, we can apply Lemma~\ref{lem1} to obtain the $LDL^T$-based factorization of $T_{0,m}(\mathcal{L}_{X})[\widehat{\Phi}_{k-1}]$.
In this case, the computational cost would include using a column-compression strategy in each recursive step
to eliminate the redundant information of the generated splitting factors. This may be very costly if the value of the scaling parameter $s$ is too large. Instead of approximating $\varphi_0(\mathcal{L}_{X})[\widehat{\Phi}_{k-1}]$
by $T_{0,m}(\mathcal{L}_{X})[\widehat{\Phi}_{k-1}]$, we consider an alternative way based on the identity
\begin{equation*}\label{3.11}
\varphi_0(\mathcal{L}_{X})[\widehat{\Phi}_{k-1}]=e^{X}\widehat{\Phi}_{k-1}e^{X^T}.
\end{equation*}
By approximating $e^{X}$ by the order $m+l$ truncated Taylor series $T_{0,m}(X),$ we can
evaluate $\varphi_0(\mathcal{L}_{X})[\widehat{\Phi}_{k-1}]$ in a low-rank format
\begin{equation*}\label{3.12}
(1-\frac{1}{k})^l\varphi_0(\mathcal{L}_{X})[\widehat{\Phi}_{k-1}]\approx \widetilde{G}_k\widetilde{S}_k\widetilde{G}_k^T,
\end{equation*}
where $\widetilde{G}_k=T_{0,m}(X)L_{k-1}$, $\widetilde{S}_k=(1-\frac{1}{k})^lD_{k-1}$.
This approach not only saves $s-1$ column compressions but also can yield more accurate results.
Finally, the splitting factors $L_k$, $D_k$ of $\widehat{\Phi}_k$ can be computed as follows:
\begin{equation*}\label{3.12a}
L_k=\left[\widetilde{G}_k, G_k \right]~\text{and}~D_k=\text{blkdiag}\left(\widetilde{S}_k,S_k\right).
\end{equation*}

The entire procedure for obtaining a low-rank approximation of $\varphi_l(\mathcal{L}_{X})[LDL^T]$ is given in Algorithm~\ref{alg3}. The computational cost of the method is dominated by the multiplications of matrices and block vectors, making it competitive for large and sparse matrix $A$.

\begin{algorithm}[h]
\caption{~$\texttt{phi\underline{~}lyap\underline{~}ldl}$:~this algorithm computes low-rank approximation of $\varphi_l(\mathcal{L}_A)[LDL^T]$.}\label{alg3}
\begin{algorithmic}[1]
\REQUIRE~$A\in \mathbb{R}^{N\times N},$ $L\in \mathbb{R}^{N\times r},$ $D\in \mathbb{C}^{r\times r},$ $l$
\STATE Call Algorithm~\ref{alg2.2} to determine the values of $m$ and $s$
\STATE $X=s^{-1}A;$
\STATE Form $\widetilde{L}_l= [L, XL, \frac{1}{2!}X^2L,\ldots,\frac{1}{m!}X^mL]$ and $\widetilde{D}_l=(\Gamma\otimes D),$ where $\Gamma$ is as defined in (\ref{3.5})
\STATE Column-compress $\widetilde{L}_l$ and $\widetilde{D}_l$
\IF {$s=1$} \RETURN $\widetilde{L}_l,$ $\widetilde{D}_l;$  \ENDIF
\FOR{$k=l-1$ to $1$}
\STATE $\widetilde{L}_k= [L,\widetilde{L}_{k+1},X\widetilde{L}_{k+1}]$
\STATE $\widetilde{D}_k=\left(\begin{tabular}{cccccc}
$\frac{1}{k!}D$ & $0$& $0$\\
$0$ & $0$& $\widetilde{D}_{k+1}$\\
$0$ & $ \widetilde{D}_{k+1}$& $0$\\
\end{tabular}%
\right)$
\STATE Column-compress $\widetilde{L}_k$ and $\widetilde{D}_k$
\ENDFOR
\STATE $L_1=\widetilde{L}_l$,~~$D_1=\widetilde{D}_l$
\FOR{$k=2:s$}
\STATE Compute $\mu_{k,j}=(1-\frac{1}{k})^{l-j}(\frac{1}{k})^j\frac{1}{(l-j)!},$ $j=1,2,\ldots,l$
\STATE Form $G_k=[\widetilde{L}_1,\widetilde{L}_2,\ldots,\widetilde{L}_l]$ and $S_k=\text{blkdiag}(\mu_{k,1}\widetilde{D}_1,\mu_{k,2}\widetilde{D}_2,\ldots,\mu_{k,l}\widetilde{D}_l)$
\STATE Compute $\widetilde{G}_k=T_{0,m}(X)L_{k-1}$ and $\widetilde{S}_k=(1-\frac{1}{k})^l D_{k-1}$
\STATE Form $L_k=[\widetilde{G}_k,G_k]$ and $D_k=\text{blkdiag}\left(\widetilde{S}_k,S_k\right)$
\STATE Column-compress $L_k$ and $D_k$
\ENDFOR
\ENSURE~$L_s$ and $D_s$
\end{algorithmic}
\end{algorithm}

\section{Numerical experiments}\label{sec:4}

In this section, we perform some numerical experiments to illustrate that the method described in Algorithm~\ref{alg3} ($\texttt{phi\underline{~}lyap\underline{~}ldl}$) can serve as the basis of a low-rank matrix-valued exponential integrator for solving DLEs and DREs. 
All the tests are conducted using MATLAB R2020b running on a desktop equipped with Intel Core i7 processor operating at 2.1GHz and 64GB of RAM.
The relative error at the time $t$ is measured in the Frobenius norm, defined as:
\begin{equation*}\label{5.1}
Error= \frac{\|\widehat{Y}_t-Y_t\|_F}{\|Y_t\|_F},
\end{equation*}
where $\widehat{Y}_t$ and $Y_t$ are the computed solution and the reference solution at the time $t$, respectively.

\begin{example} \label{exa1a} 
Consider the two-dimensional heat equation (HE) with a Gaussian source term:
\begin{equation}\label{5.2a}
u_t= \alpha\Delta u+ \text{exp}{\left(-\frac{(x-\mu)^2+(y-\mu)^2}{2\sigma^2}\right)},  \quad (x, y)\in \Omega=[0, d]^2, \quad t>0
\end{equation}
 with Dirichlet boundary conditions $u = 0$ on $\partial \Omega$. We choose initial condition 
 \begin{equation*}\label{5.2b}
u(x,y,0)=\sin(\pi x) \sin(\pi y).
\end{equation*}
\end{example} 

We discretize  Eq. (\ref{5.2a}) in space by standard finite differences, employing $N$ interior grid points in each direction. This results in a mesh size of $h=\frac{1}{N+1}$ and produces the following  differential Lyapunov equations (DLEs):
\begin{equation}\label{5.2c}
\left\{
\begin{array}{l}
U'(t)=AU(t)+U(t)A^T+BB^T,\\
U(0)=L_0L_0^T,
\end{array}
\right.
\end{equation}
where
 $A= \alpha(n+1)^2\text{tridiag}(1,-2,1)\in \mathbb{R}^{N\times N}$, $B, L_0\in  \mathbb{R}^{N}$ with $[B]_i=\text{exp}{\left(-\frac{(ih-\mu)^2}{2\sigma^2}\right)}$ and  
$[L_0]_i=\sin(\pi ih)$, for $i=1,2,\ldots,N$.

The solution to the DLEs (\ref{5.2c}) at time $t$ is exactly represented using the \texttt{mExpeul} scheme:
\begin{equation*}
U(t)=e^{t\mathcal{L}_A}[L_0L_0^T]+t\varphi_1(t\mathcal{L}_A)[BB^T]=U(t_0)+t\varphi_1(t\mathcal{L}_A)[F(U(t_0))].
\end{equation*}
The scheme is implemented by first decomposing the term $F(U(t_0))$ into an $LDL^T$-type factorization:
\begin{eqnarray*}
\begin{array}{llll}
F(U(t_0)) &=AU_0+U_0A^T-BB^T\\
&=[L_0, AL_0, B]\left(\begin{tabular}{cccccc}
$0$ & $I$& $0$\\
$I$ & $0$& $0$\\
$0$ & $0$& $I$\\
\end{tabular}%
\right)[L_0, AL_0, B]^T\\
&=:L_fD_fL_f^T.\\
\end{array}
\end{eqnarray*}
Next, the $\texttt{phi\underline{~}lyap\underline{~}ldl}$ function is applied
to compute the $LDL^T$ factorization of $t\varphi_1(t\mathcal{L}_A)[L_fD_fL_f^T]$.
Finally, the low-rank factors of $X(t)$ are formed by combing the two terms on right hand side. The resulting algorithm for the \texttt{mExpeul} is given
in Algorithm \ref{mExpeul}.
\begin{algorithm}[H]
\caption{~Low-rank implementation of \texttt{mExpeul} scheme for DLEs (\ref{5.2c})} \label{mExpeul}
\begin{algorithmic}[1]
\REQUIRE ~$A\in \mathbb{R}^{N\times N}$, $B\in \mathbb{R}^{ N\times q}$, $L_0\in \mathbb{R}^{ N\times r},$ $t$.
\STATE Form $L_f=[L_0, AL_0, B]$ and $D_f=\left(\begin{tabular}{cccccc}
$0$ & $I_r$& $0$\\
$I_r$ & $0$& $0$\\
$0$ & $0$& $I_q$\\
\end{tabular}%
\right).$ 
\STATE Column-compress $L_f$ and $D_f$.
\STATE Compute the low-rank factors  $\tilde{L}_f$, $\tilde{D}_f$ of $\varphi_1(t\mathcal{L}_A)[L_fD_fL_f^T]$ using Algorithm \ref{alg3}.
\STATE Form $L_t=[L_0,\tilde{L}_f]$ and $D_t=\text{blkdiag}(I_r, t\tilde{D}_f)$.
\STATE Column-compress $L_t$ and $D_t$.
\ENSURE~$L_t$ and $D_t$.
\end{algorithmic}
\end{algorithm}

We integrate the problem (\ref{5.2c}) with the parameter values $\alpha =0.02$, $\mu= 5$, $\sigma=1$, $d=10$,
and $N=1000$. In our implementation, the compression accuracy of \texttt{mExpeul} is set to $100\cdot eps$. 
 Fig. \ref{fig5.1a} (a)-(d) presents the solution obtained with \texttt{mExpeul} and its difference from the reference solution. The results for $t=1$ are shown on the left, and those for $t=5$ are on the right. The reference solution is computed by rewriting (\ref{5.2c}) as a vector ODEs of dimension $N^2$ using the  Kronecker product and then solving it with the vector-valued exponential Euler scheme (denoted by \texttt{vExpeul}) \cite{Hochbruck2010}. The implementation of \texttt{vExpeul} requires the evaluation of the matrix $\varphi$-function, which is performed by krylov-based method \texttt{phipm} \cite{Niesen2012}. The relative errors of \texttt{mExpeul} with respect to the  reference solution are 2.4571\text{e}-14 at $t=1$ and 4.6354\text{e}-13 at $t=5$. The computation times for \texttt{mExpeul} are 0.40 seconds at $t=1$ and 2.03 seconds at $t=5$, whereas  \texttt{vExpeul} requires 9.11 seconds at $t=1$ and 36.81 seconds at $t=5$.

\begin{figure}[H]
\begin{minipage}{0.5\linewidth}
\centering
\includegraphics[width=6.5cm,height=5cm]{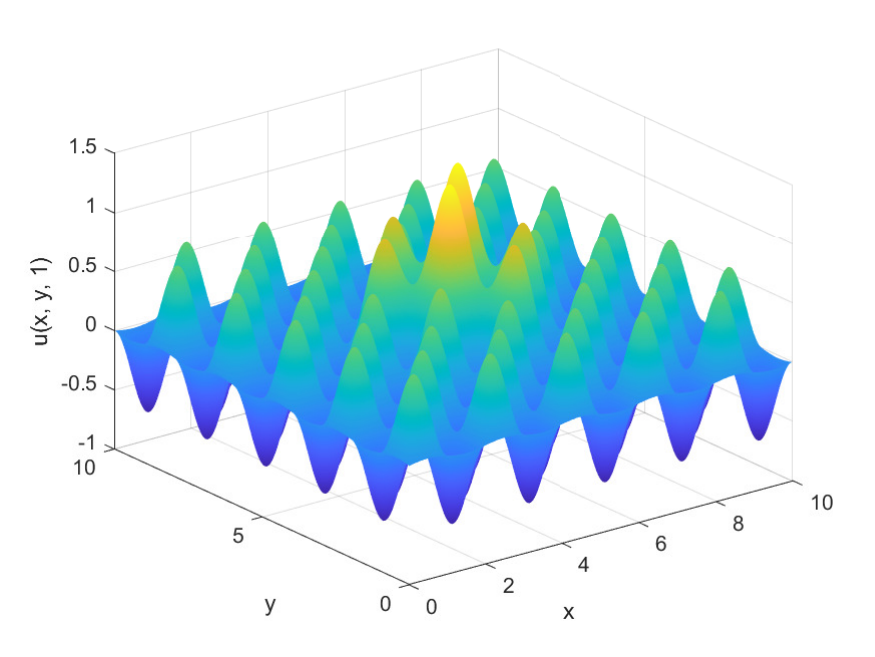}\\
\footnotesize{Solution of \texttt{mExpeul} at $t=1$}
\end{minipage}
\mbox{\hspace{-1.5cm}}
\begin{minipage}{0.5\linewidth}
\centering
\includegraphics[width=6.5cm,height=5cm]{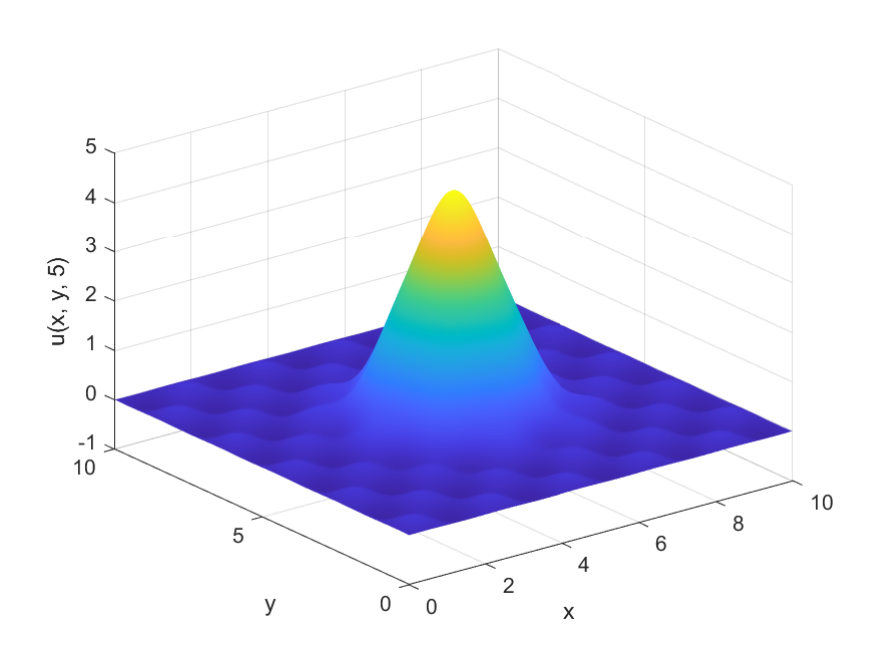}\\
\footnotesize{Solution of \texttt{mExpeul} at $t=5$ }
\end{minipage}\\
\begin{minipage}{0.5\linewidth}
\centering
\includegraphics[width=6.5cm,height=5cm]{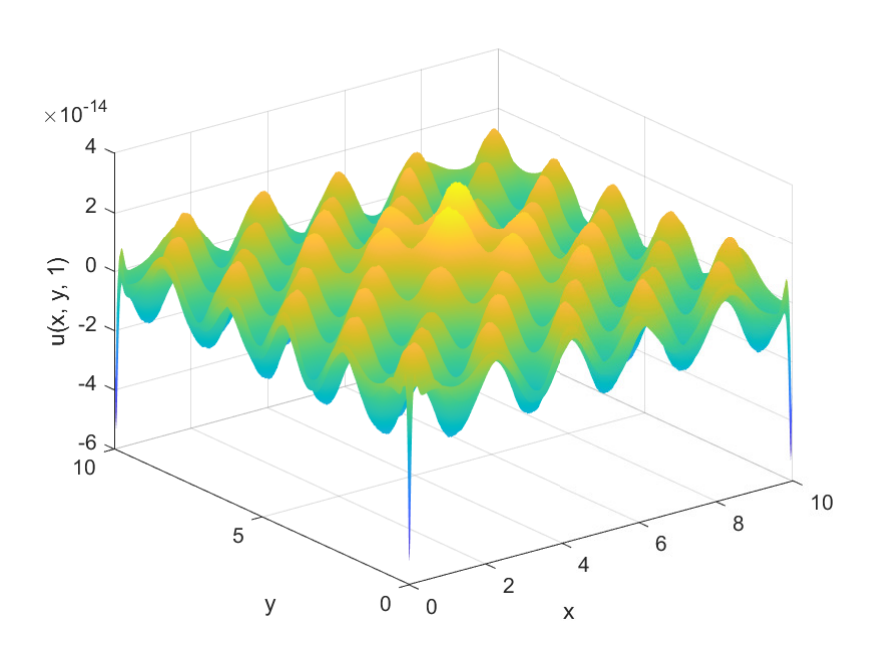}\\
\footnotesize{Difference  between \texttt{mExpeul} and \texttt{vExpeul} at $t=1$}
\end{minipage}
\mbox{\hspace{-1.5cm}}
\begin{minipage}{0.5\linewidth}
\centering
\includegraphics[width=6.5cm,height=5cm]{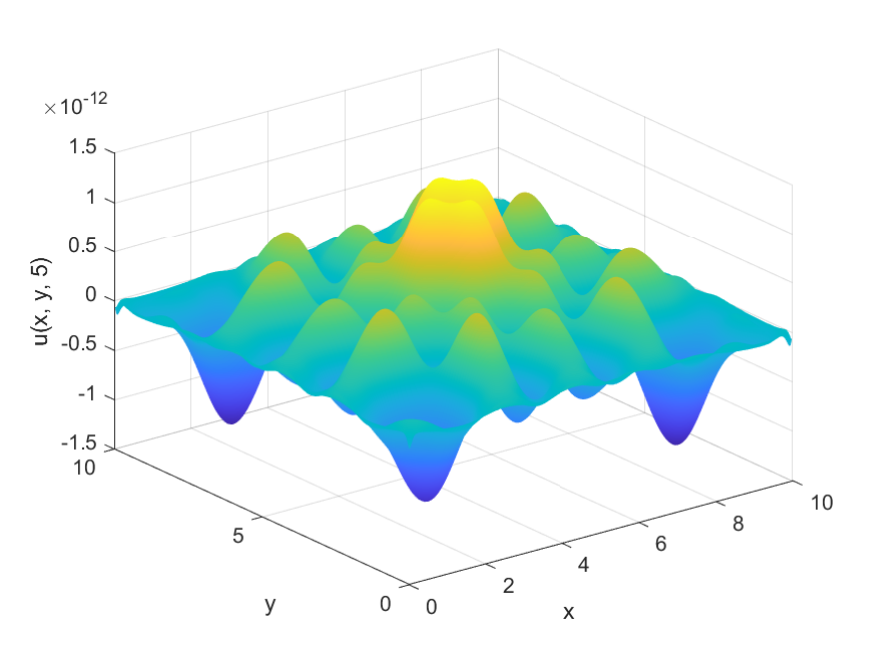}\\
\footnotesize{Difference  between \texttt{mExpeul} and \texttt{vExpeul} at $t=5$}
\end{minipage}
\caption{Results of \texttt{mExpeul} at $t=1$ (left) and $t=5$ (right) for Example \ref{exa1a}, respectively.}\label{fig5.1a}
\end{figure}

\begin{example} 
\label{exa1}
In the second example, we again consider  DLEs of the form (\ref{5.2c}). Here, the coefficient matrix $A\in\mathbb{R}^{N\times N}$ is generated using the same approach as in \cite{Mena}, which arises from the spatial finite difference discretization of
the two-dimensional heat equation
\begin{equation*}
\frac{\partial u}{\partial t}=\alpha \Delta u,~~u|_{\partial\Omega}=0
\end{equation*}
on $\Omega=(0,1)^2$ with $N = n^2$ inner discretization nodes
and mesh size $h = \frac{1}{n+1}$. Thus, we have
\begin{equation*}
 A =\alpha (n+1)^2(I_n\otimes K  + K\otimes I_n),~~K= \text{tridiag}(1,-2,1)\in\mathbb{R}^{n\times n}.
\end{equation*}
The low-rank matrices $B\in  \mathbb{R}^{N\times 5}$ and
$L_0\in  \mathbb{R}^{N\times 2}$ are randomly generated by MATLAB function \texttt{randn}, respectively.
\end{example}

To evaluate the performance of $\texttt{phi\underline{~}lyap\underline{~}ldl}$,
we compare \texttt{mExpeul} with two classes of low-rank integrators introduced in \cite{Lang2016}. These include
backward differentiation formula (BDF)  methods of orders 1 to 5 and Rosenbrock methods of orders 1 and 2. For simplicity,  we abbreviate them as \texttt{BDF1}, \texttt{BDF2}, \texttt{BDF3}, \texttt{BDF4},  \texttt{BDF5}, \texttt{Ros1}, and \texttt{Ros2}, respectively. In this example, the compression accuracy \texttt{Tol} of \texttt{mExpeul} is set as $N\cdot eps$.

In this test, we take $n=100$, resulting in $N=n^2=10^4$. Numerical results are provided for three 
 different values of the coefficients $\alpha=2\cdot10^{-4}$, $2\cdot10^{-3}$, and $2\cdot10^{-2}.$  As the value of $\alpha$ increases,  the stiffness of the problem also grows. The BDF methods and Rosenbrock methods are implemented with  
a constant time step size of $h= 0.01$. Due to memory limitations, it is not feasible to use \texttt{vExpeul} to compute the reference solution for this problem. Instead, \texttt{BDF3} is employed with a smaller time step of $h=10^{-4}$ to calculate the reference solutions.

Table~\Ref{tab5.2} summarizes  the performance of the methods in terms of accuracy and CPU time at the integration time $t = 1$.
The results show  that \texttt{mExpeul}, which leverages $\texttt{phi\underline{~}lyap\underline{~}ldl}$, outperforms the other methods
in efficiency. Theoretically, \texttt{mExpeul} is capable of generating approximations with arbitrary accuracy for DLEs.
However, due to the limitations in the accuracy of the reference solution produced, the errors of \texttt{mExpeul} reported 
in Table~\ref{tab5.2} should be regarded as a conservative estimate.

We also observe that the impact of all the methods on efficiency becomes particularly evident as 
the stiffness of the problem increases. For \texttt{mExpeul}, the norm of Lyapunov operator grows with
the problem's stiffness, which would, in turn, lead to a larger scaling parameter $s$
in $\texttt{phi\underline{~}lyap\underline{~}ldl}$ and increases the computational load. The performance could potentially be improved in the future by employing preprocessing techniques to reduce the norm of Lyapunov operator.

\begin{table}[h]
\caption{The relative errors and the CPU times (in seconds) for Example~\ref{exa1}}\label{tab5.2}
\begin{tabular*}{\textwidth}{@{\extracolsep\fill}lcccccc}
\toprule%
& \multicolumn{2}{@{}c@{}}{$\alpha=2\cdot10^{-4}$} & \multicolumn{2}{@{}c@{}}{$\alpha=2\cdot10^{-3}$}& \multicolumn{2}{@{}c@{}}{$\alpha=2\cdot10^{-2}$} \\\cmidrule{2-3}\cmidrule{4-5} \cmidrule{6-7}%
Methods &Error &Time &Error &Time &Error&Time  \\
\midrule
\texttt{BDF1} &2.4790e-03 &50.89    &1.6985e-03 &92.04   &1.4036e-03 &261.90\\
\texttt{BDF2}&4.9224e-05&47.73   &3.2074e-05&96.38   &2.7815e-05&245.92   \\
\texttt{BDF3} &6.3121e-06&50.85    &4.1075e-06&110.57   &3.5405e-06&314.05    \\
\texttt{BDF4} &6.5234e-06&55.04    &4.3161e-06&115.55   &3.7161e-06& 376.30   \\
\texttt{BDF5} &8.4644e-06& 60.08  &5.5856e-06&148.00   &4.8154e-06&560.94 \\
\texttt{Ros1}&2.4790e-03&45.75   &1.6985e-03&81.07  &1.4036e-03&195.12   \\
\texttt{Ros2}&1.8896e-04&304.34    &1.2398e-04&1519.27   &1.0741e-04&6266.49  \\
\texttt{mExpeul}&1.1435e-09 &1.93    &9.6709e-08 &9.38 &3.5272e-09&112.60 \\
\bottomrule
\end{tabular*} 
\end{table}

\begin{figure}[H]
\begin{minipage}{0.5\linewidth}
\centering
\includegraphics[width=6.5cm,height=5cm]{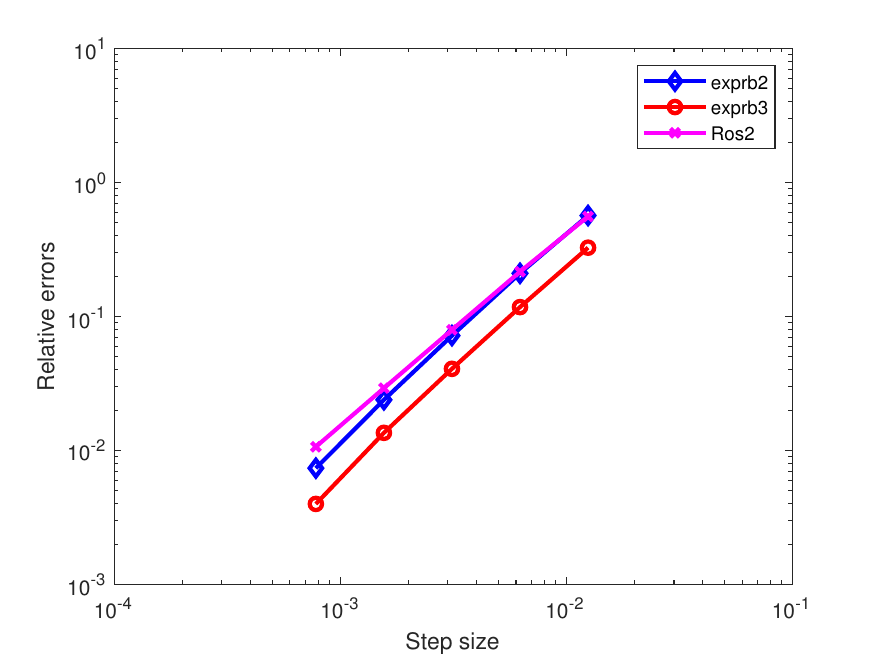}\\
\footnotesize{$N=1600$}
\end{minipage}
\mbox{\hspace{-1.5cm}}
\begin{minipage}{0.5\linewidth}
\centering
\includegraphics[width=6.5cm,height=5cm]{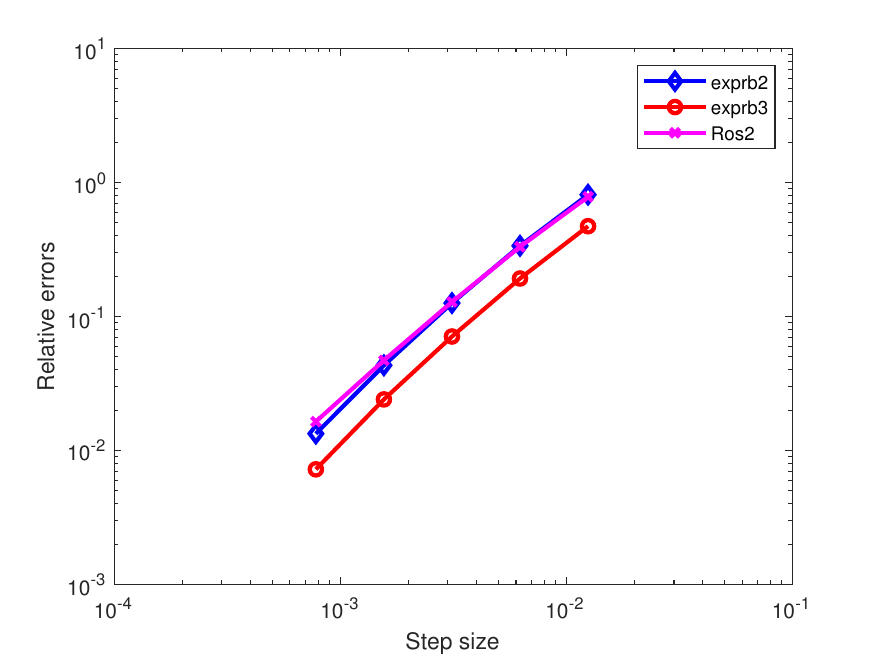}\\
\footnotesize{$N=2500$}
\end{minipage}\\
\begin{minipage}{0.5\linewidth}
\centering
\includegraphics[width=6.5cm,height=5cm]{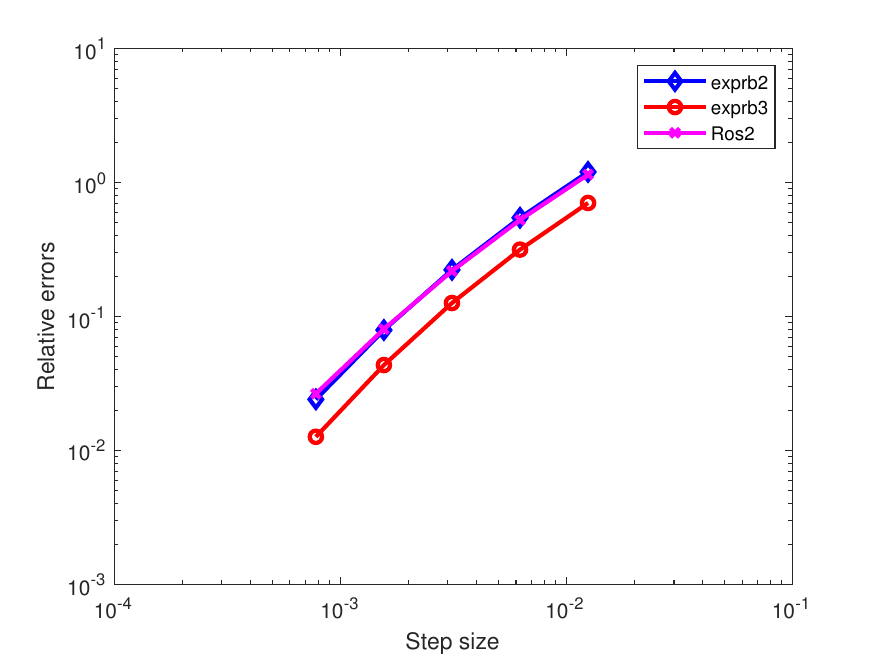}\\
\footnotesize{$N=3600$}
\end{minipage}
\mbox{\hspace{-1.5cm}}
\begin{minipage}{0.5\linewidth}
\centering
\includegraphics[width=6.5cm,height=5cm]{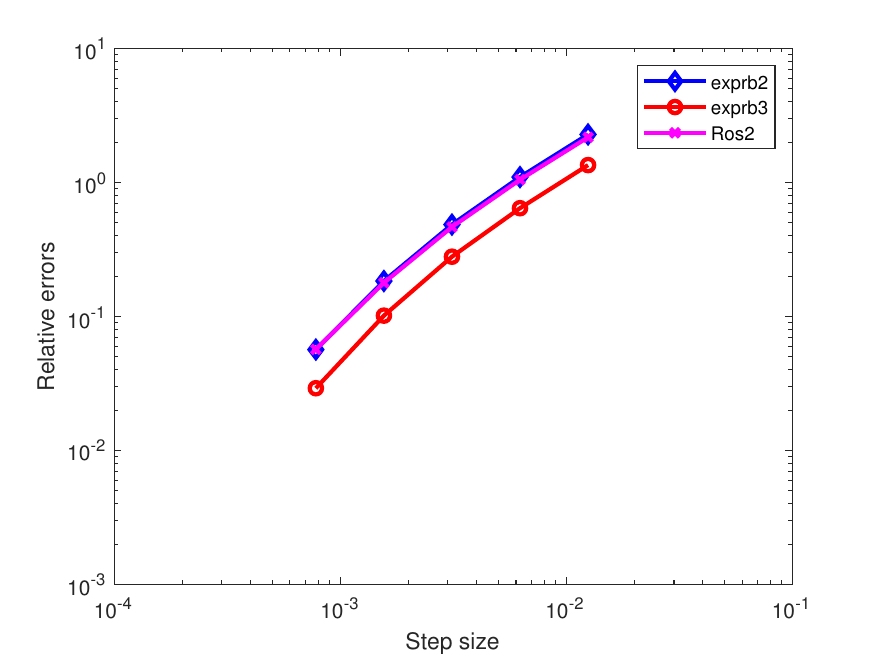}\\
\footnotesize{$N=6400$}
\end{minipage}
\caption{Relative errors of \texttt{exprb2}, \texttt{exprb3}, and \texttt{Ros2} under the discrete $L^2(0,0.1,\mathbb{R}^{N\times N})$-norm versus
various time step sizes for the integration of each equation in Example~\ref{exa2}.}\label{fig4.1a}
\end{figure}

\begin{figure}[tbhp]
\begin{minipage}{0.5\linewidth}
\centering
\includegraphics[width=6.5cm,height=5cm]{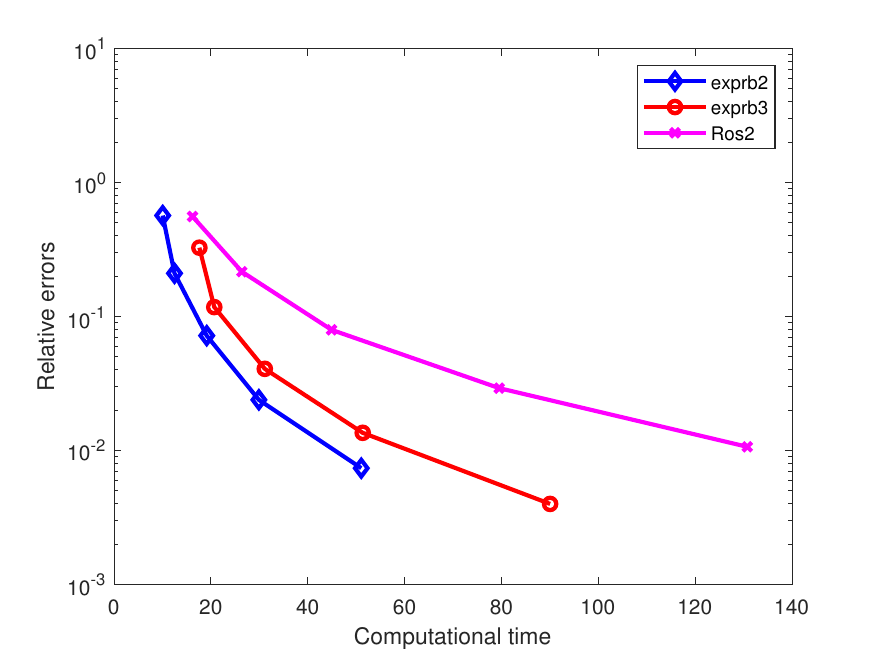}\\
\footnotesize{$N=1600$}
\end{minipage}
\mbox{\hspace{-1.5cm}}
\begin{minipage}{0.5\linewidth}
\centering
\includegraphics[width=6.5cm,height=5cm]{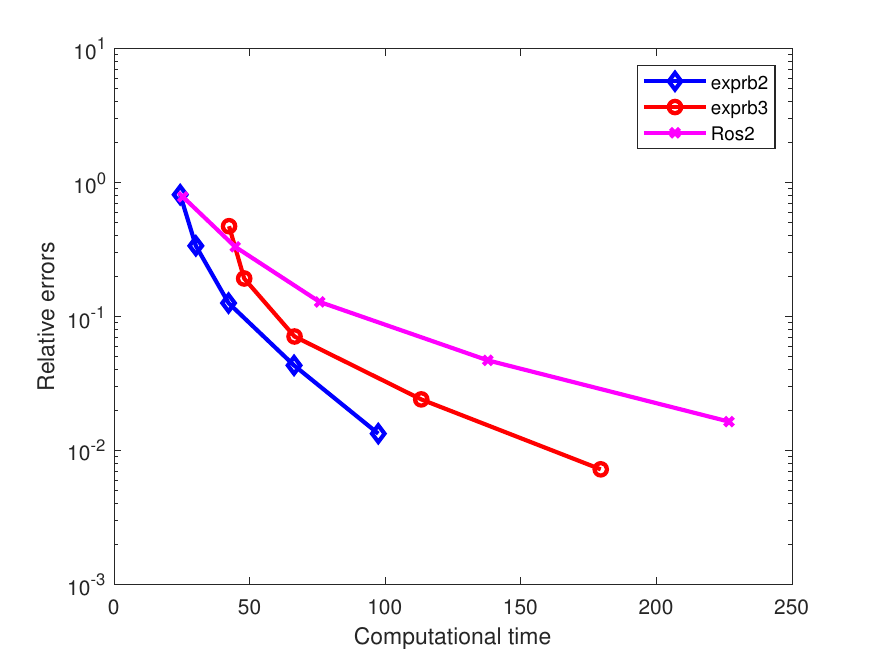}\\
\footnotesize{$N=2500$}
\end{minipage}\\
\begin{minipage}{0.5\linewidth}
\centering
\includegraphics[width=6.5cm,height=5cm]{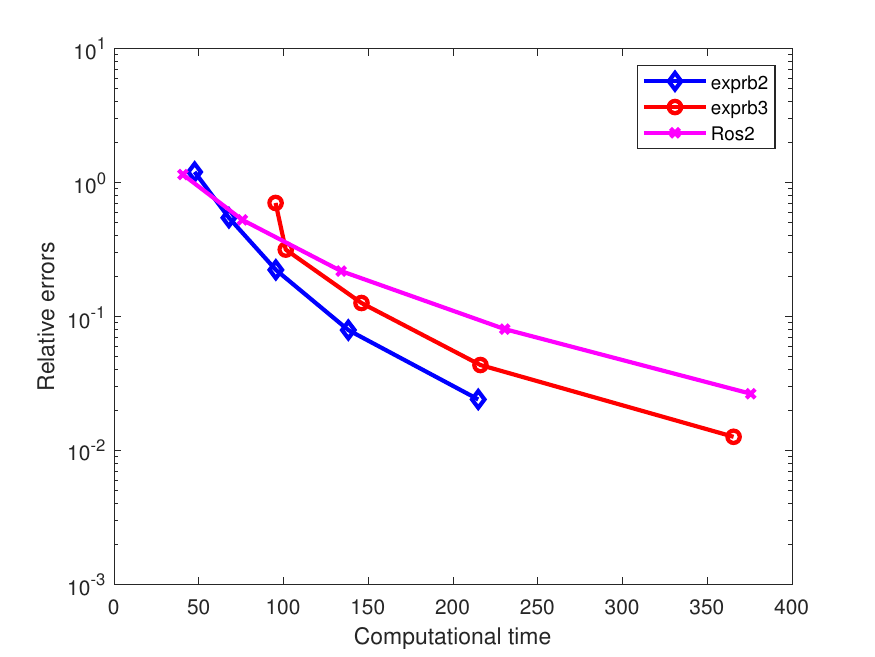}\\
\footnotesize{$N=3600$}
\end{minipage}
\mbox{\hspace{-1.5cm}}
\begin{minipage}{0.5\linewidth}
\centering
\includegraphics[width=6.5cm,height=5cm]{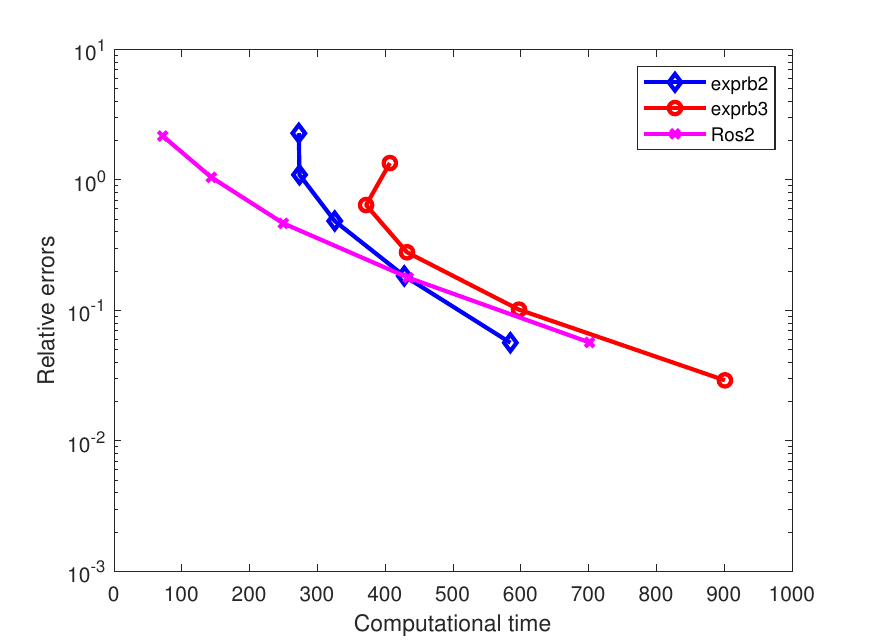}\\
\footnotesize{$N=6400$}
\end{minipage}
\caption{Relative errors of \texttt{exprb2}, \texttt{exprb3}, and \texttt{Ros2} under the discrete $L^2(0,0.1,\mathbb{R}^{N\times N})$-norm
versus computation times for the integration of each equation in Example~\ref{exa2}.}\label{fig4.1b}
\end{figure}

\begin{figure}[tbhp]
\begin{minipage}{0.5\linewidth}
\centering
\includegraphics[width=6.5cm,height=5cm]{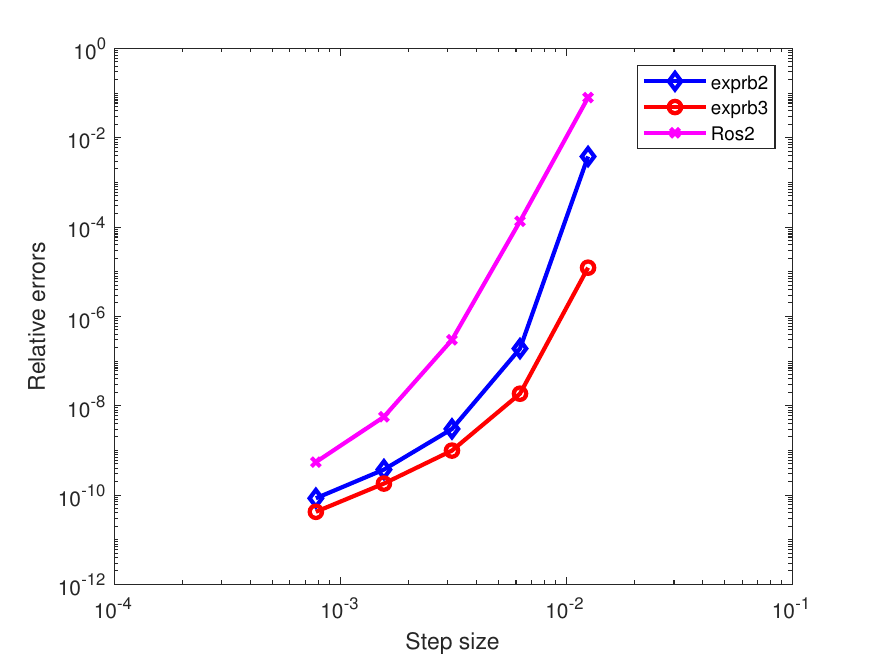}\\
\footnotesize{$N=1600$}
\end{minipage}
\mbox{\hspace{-1.5cm}}
\begin{minipage}{0.5\linewidth}
\centering
\includegraphics[width=6.5cm,height=5cm]{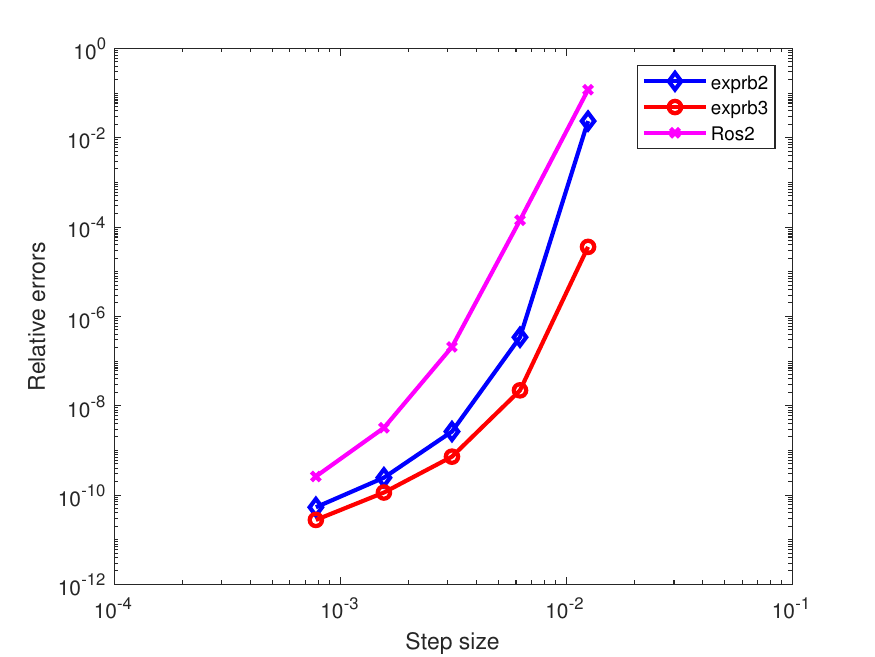}\\
\footnotesize{$N=2500$}
\end{minipage}\\
\begin{minipage}{0.5\linewidth}
\centering
\includegraphics[width=6.5cm,height=5cm]{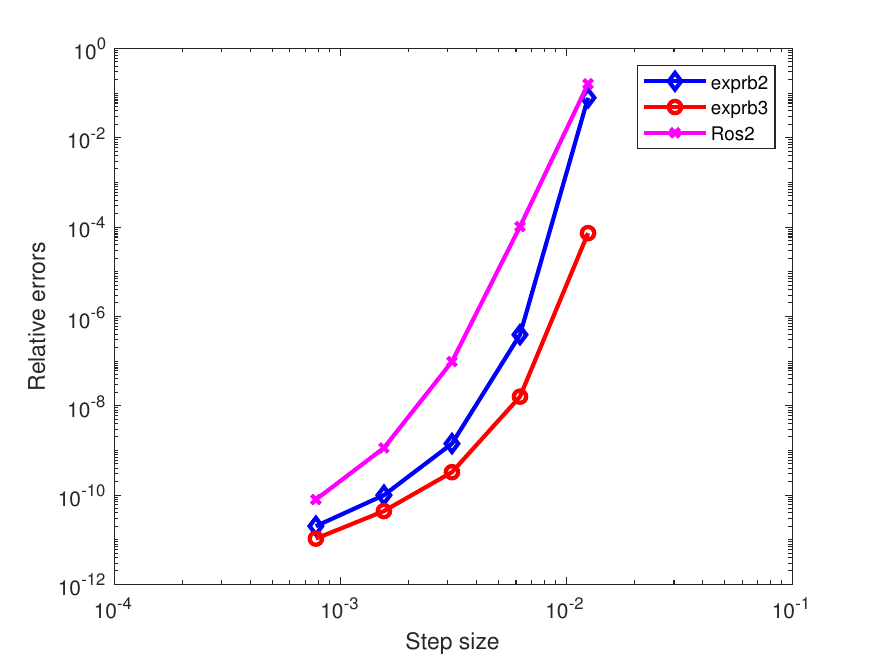}\\
\footnotesize{$N=3600$}
\end{minipage}
\mbox{\hspace{-1.5cm}}
\begin{minipage}{0.5\linewidth}
\centering
\includegraphics[width=6.5cm,height=5cm]{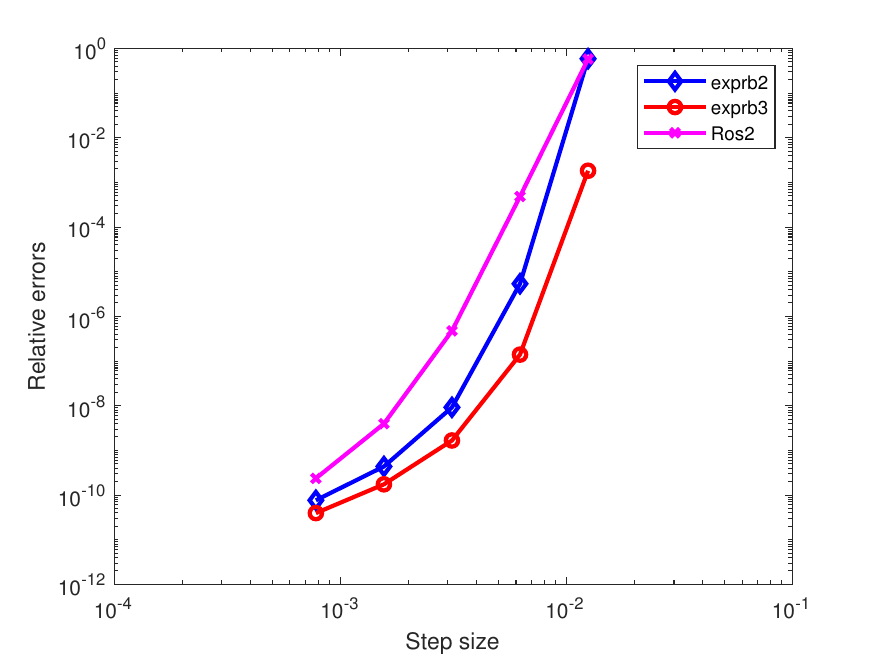}\\
\footnotesize{$N=6400$}
\end{minipage}
\caption{Relative errors of \texttt{exprb2}, \texttt{exprb3}, and \texttt{Ros2} at $T=0.1$ versus
various time step sizes for the integration of each equation in Example~\ref{exa2}.}\label{fig4.2a}
\end{figure}

\begin{figure}[tbhp]
\begin{minipage}{0.5\linewidth}
\centering
\includegraphics[width=6.5cm,height=5cm]{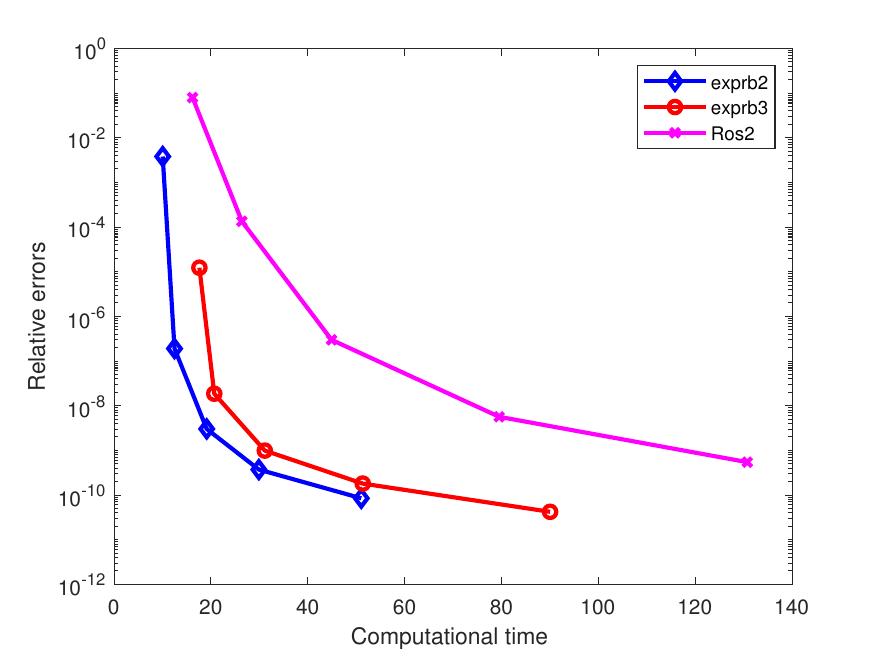}\\
\footnotesize{$N=1600$}
\end{minipage}
\mbox{\hspace{-1.5cm}}
\begin{minipage}{0.5\linewidth}
\centering
\includegraphics[width=6.5cm,height=5cm]{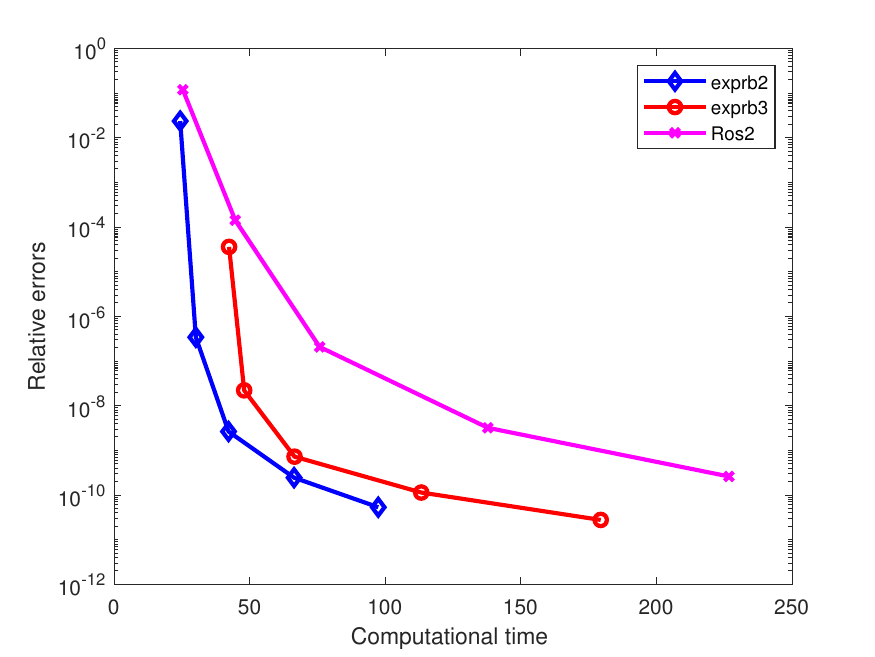}\\
\footnotesize{$N=2500$}
\end{minipage}\\
\begin{minipage}{0.5\linewidth}
\centering
\includegraphics[width=6.5cm,height=5cm]{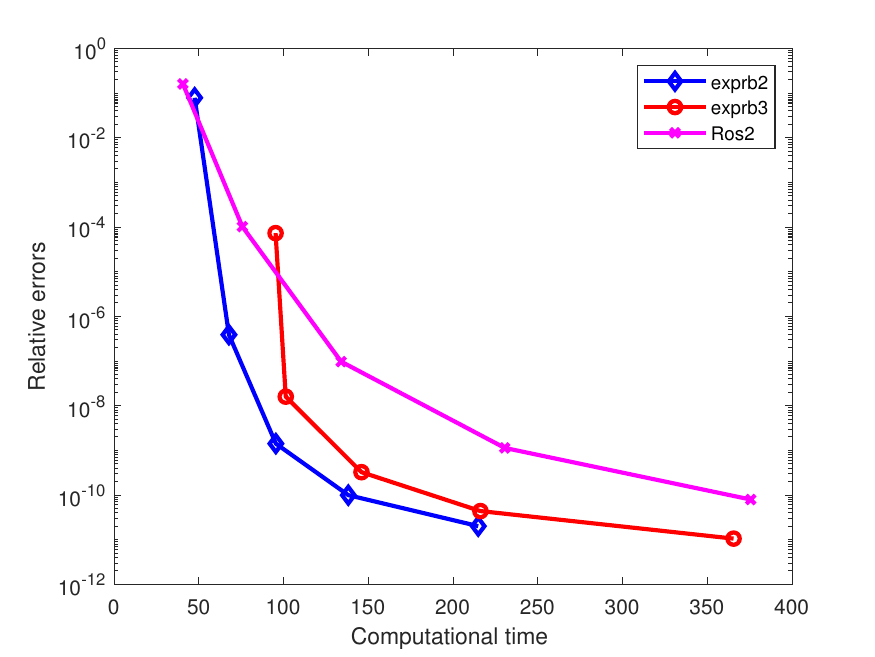}\\
\footnotesize{$N=3600$}
\end{minipage}
\mbox{\hspace{-1.5cm}}
\begin{minipage}{0.5\linewidth}
\centering
\includegraphics[width=6.5cm,height=5cm]{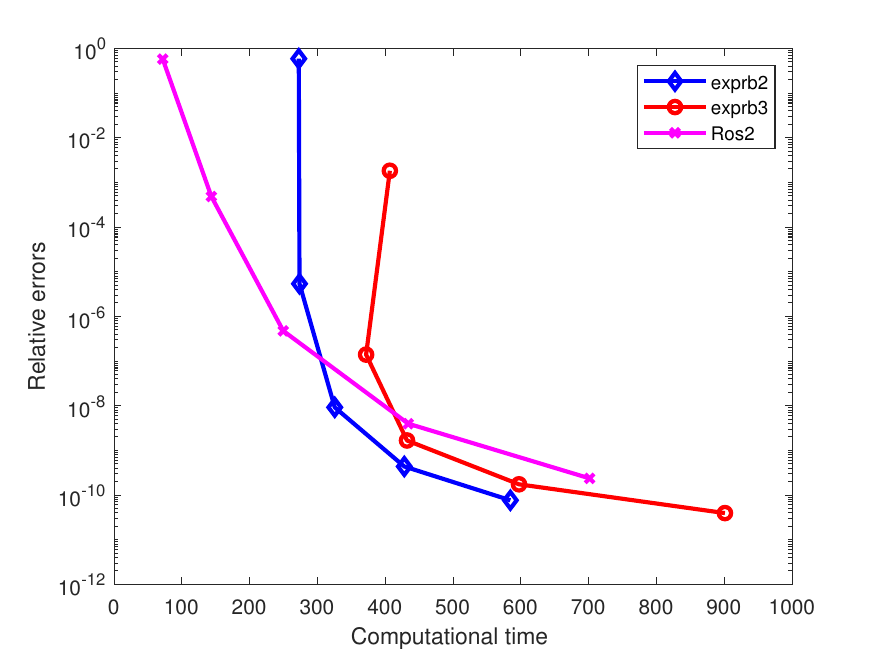}\\
\footnotesize{$N=6400$}
\end{minipage}
\caption{Relative errors of \texttt{exprb2}, \texttt{exprb3}, and \texttt{Ros2} at $T=0.1$
versus computation times for the integration of each equation in Example~\ref{exa2}.}\label{fig4.2b}
\end{figure}

\begin{example} \label{exa2}
Consider the DREs \cite{Penzl}
\begin{equation*}\label{5.2}
\left\{
\begin{array}{l}
X'(t)=AX(t)+X(t)A^T+C^T C-X(t)BB^TX(t),\\
X(0)=L_0L_0^T.
\end{array}
\right.
\end{equation*}
The coefficient matrix $A\in \mathbb{R}^{N\times N}$
results from the 5-point finite difference discretization of the advection-diffusion equation
\begin{equation*}
\frac{\partial u}{\partial t}=\Delta u - 10x \frac{\partial u}{\partial x} - 100y \frac{\partial u}{\partial y}
\end{equation*}
on the unit square $\Omega=(0,1)^2$ with homogeneous Dirichlet boundary conditions. The vectors $B, C^T\in \mathbb{R}^{N}$ serve as the corresponding load vectors.
These matrices are generated using MATLAB routines \texttt{fdm\underline{~}2d\underline{~}matrix} and \texttt{fdm\underline{~}2d\underline{~}vector} from LyaPack \cite{Penzl}, where  $A=\texttt{fdm\underline{~}2d\underline{~}matrix}(n_0,'10*x','100*y','0')$, \quad $B=\texttt{fdm\underline{~}2d\underline{~}vector}(n_0,'0.1< x\leq 0.3')$, \quad $C^T=\texttt{fdm\underline{~}2d\underline{~}vector}(n_0,'0.7< x\leq 0.9')$. Here, $n_0$ denotes the number of grid points in each spatial direction, and $N=n_0^2$. 
The low-rank factor $L_0\in\mathbb{R}^{N\times 1} $ of the initial value $X(0)$ is generated randomly.
\end{example}

This problem is a commonly used benchmark test. We simulate it over the time interval $[0, 0.1]$ using the time step
sizes $h\in\{\frac{1}{80}, \frac{1}{160}, \frac{1}{320}, \frac{1}{640}, \frac{1}{1280}\}$,  which correspond to $N_h\in\{8, 16, 32, 64, 128\}$ time steps, respectively.

As in \cite[Experiment 3]{Li2021}, we employ two low-rank matrix-valued exponential integration schemes, \texttt{exprb2} and \texttt{exprb3},
 to integrate the problem for four different dimensions $N=1600,2500,3600,6400$, respectively. The first scheme is second-order and involves the operator function $\varphi_1$, while the second scheme is third-order and requires the evaluation of both $\varphi_1$ and $\varphi_3$ at each time step. In \cite{Li2021}, 
these operator functions are computed using the numerical quadrature method. In contrast, we now use $\texttt{phi\underline{~}lyap\underline{~}ldl}$ to evaluate them.
 
For a comparison of the computational times and relative errors with respect to a reference solution, the two exponential integration schemes are also compared against the second-order Rosenbrock-type method (denoted by \texttt{Ros2}) as developed in M.E.S.S. toolbox \cite{Saak}. The common feature of these two types of is their ability to preserve the equilibrium point of the system. Specifically, this is reflected in the fact that the closer the integration process gets to the system's equilibrium point, the higher the accuracy of the algorithm. In our test, the \texttt{Ros2} method is executed using its default parameters. All the reference solutions are produced by \texttt{Ros2} with a finer time step size of $h=1/10240.$

Figs.~\ref{fig4.1a} and \ref{fig4.1b} present the accuracy and efficiency plots, respectively, for three methods (\texttt{exprb2}, \texttt{exprb3} and \texttt{Ros2}) under the discrete $L^2(0,0.1,\mathbb{R}^{N\times N})$-norm defined as
\begin{equation*}
Error= \sum\limits^{ N_h}_{k=1}h \frac{\|X(kh)-X_k\|_F}{\|X(kh)\|_F}, 
 \end{equation*}
where $X(kh)$ and $X_k$ are the reference solution and the numerical solution at the time $t=kh$, respectively.
The results demonstrate that all three methods achieve their expected convergence orders asymptotically, with \texttt{exprb2} and \texttt{Ros2} exhibiting nearly identical accuracy curves. The efficiency plots indicate that the matrix-valued exponential integrators are generally more efficient than \texttt{Ros2} in most cases.

Figs.~\ref{fig4.2a} and \ref{fig4.2b} present the accuracy and efficiency plots of the three integrators at the final time point $T=0.1$.
The results demonstrate  that \texttt{exprb2} and \texttt{exprb3} are generally more accurate and more efficient than \texttt{Ros2} for the same time step size.

Furthermore, Figs. \ref{fig4.1a} and \ref{fig4.2a} reveal significant differences in the error of the matrix-valued exponential integrators under two distinct error metrics. Specifically, the errors measured using the $L^2(0,0.1,\mathbb{R}^{N\times N})$-norm are substantially larger than those observed at the final time point $T=0.1$. This indicates that these methods incur greater errors during transient states. A potential improvement could involve developing adaptive matrix-valued exponential integrators, enabling the use of smaller time steps near transient states and larger time steps closer to steady-state solutions. The construction of such adaptive methods, however, is beyond the scope of this paper and will be addressed in future research.

The M.E.S.S. toolbox also provides two additional methods for computing low-rank solutions of DREs: the BDF method and the splitting method. The results of the BDF method are not included here because it fails to produce convergent solutions in some cases, despite its capability to preserve the system's steady-state solutions. In contrast, the splitting method exhibits different behavior compared to the three methods discussed earlier. Possibly due to its smaller leading coefficient of the local truncation error, it produces smaller errors near transient states. However, the other methods generally exhibit higher accuracy near steady-state solutions. Regarding computational efficiency, the splitting method may be better suited for comparison with adaptive exponential integrators, a subject for future research.

\section{Conclusion}\label{sec:5}

We have developed a low-rank algorithm to compute the Lyapunov operator 
$\varphi$-functions arising from matrix-valued exponential integrators. The algorithm's performance is evaluated by comparing matrix-valued exponential integrators against several state-of-the-art methods. Numerical results confirm the method's effectiveness and reliability, highlighting its potential as a robust foundation for solving large-scale DLEs and DREs. Building on this method, we plan to develop adaptive matrix-valued exponential integrators and evaluate their performance on more matrix differential equations.
The proposed method can be extended to compute the $\varphi$-functions of Sylvester operators, broadening its applicability. Future work will also focus on enhancing the algorithm's performance through preprocessing techniques such as shifting and balancing strategies.

\section*{Acknowledgements}
 This work of Dongping Li was supported by the Jilin Scientific and Technological Development Program (Grant No. YDZJ202501ZYTS635), the Natural Science Foundation of Jilin Province (Grant No. JJKH20240999KJ) and the Natural Science Foundation of China (Grant No. 12371455). The work of Hongjiong Tian is supported by the National Natural Science Foundation of China (Grant No. 12271368), the Science and Technology Innovation Plan of Shanghai (Grant No. 20JC1414200) and Shanghai Rising-Star Program (Grant No. 22QA1406900).

\bibliographystyle{elsarticle-num}
\bibliography{reference}

\end{document}